\documentclass[reqno]{amsart}
\usepackage{}
\usepackage{stmaryrd}
\usepackage{mathrsfs}
\usepackage{cases}
\usepackage{amsfonts}
\usepackage{amssymb}
\usepackage{amsmath}
\usepackage{tikz}
\newtheorem{theorem}{Theorem}[section]
\newtheorem{lemma}[theorem]{Lemma}

\newtheorem{corollary}[theorem]{Corollary}
\theoremstyle{definition}
\newtheorem{definition}[theorem]{Definition}

\newtheorem{remark}[theorem]{Remark}
\numberwithin{equation}{section}

\newcommand{\blankbox}[2]

\let\al=\alpha
\let\b=\beta
\let\g=\gamma
\let\d=\delta

\let\la=\lambda
\let\r=\rho
\let\s=\sigma

\let\G= \Gamma

\let\th=\theta

\let\ep=\epsilon

\let\k=\kappa
\let\fy=\infty
\def\bbR{\mathbb{R}}
\def\bbS{\mathbb{S}}

\def\bbZ{\mathbb{Z}}

\def\calC{\mathcal {C}}

\def\scrF{\mathscr{F}}

\newcommand{\be}{\begin{equation*}}
\newcommand{\ee}{\end{equation*}}
\newcommand{\ben}{\begin{equation}}
\newcommand{\een}{\end{equation}}
\newcommand{\bn}{\begin{enumerate}}
\newcommand{\en}{\end{enumerate}}
\newcommand{\bs}{\backslash}

\def\rr{{\mathbb R}}

\def\rn{{{\rr}^n}}

\def\mpq{M^{p,q}}
\def\mpqs{M^{p,q}_s}
\def\wpq{W^{p,q}}
\def\wpqs{W^{p,q}_s}
\def\-fl1{\scrF^{-1}L^1}
\def\fl1{\scrF L^1}
\def\lan{\langle}
\def\ran{\rangle}
\def\hes{\text{Hess}}

\def\wyynp{W^{\fy,\fy}_{n/\dot{p}+\ep}}
\def\zn{\mathbb{Z}^n}
\def\scrL{\mathscr{L}}

\begin{document}
\title[unimodular Fourier multipliers on Wiener amalgam spaces]
{sharp estimates of unimodular Fourier multipliers on Wiener amalgam spaces}
\author{WEICHAO GUO}
\address{School of Science, Jimei University, Xiamen, 361021, P.R.China}
\email{weichaoguomath@gmail.com}
\author{GUOPING ZHAO}
\address{School of Applied Mathematics, Xiamen University of Technology,
Xiamen, 361024, P.R.China} \email{guopingzhaomath@gmail.com}

\begin{abstract}
We study the boundedness on the Wiener amalgam spaces $\wpqs$ of
Fourier multipliers with symbols of the type $e^{i\mu(\xi)}$,
for some real-valued functions $\mu(\xi)$ whose prototype is $|\xi|^{\b}$ with $\b\in (0,2]$.
Under some suitable assumptions on $\mu$, we give the characterization of
$\wpqs\rightarrow \wpq$ boundedness of $e^{i\mu(D)}$, for arbitrary pairs of
$0< p,q\leq \infty$.
Our results are an essential improvement of the previous known results, for both sides of sufficiency and necessity,
even for the special case $\mu(\xi)=|\xi|^{\b}$ with $1<\b<2$.
\end{abstract}

\maketitle


\section{INTRODUCTION}
The main aim of this paper is to study certain unimodular Fourier
multipliers on the Wiener amalgam spaces.
For two function spaces $X$ and $Y$, we call a tempered distribution $m$ a Fourier
multiplier from $X$ to $Y$, if there exists a constant $C>0$ such that
\begin{equation*}
\Vert T_{m}(f)\Vert _{Y}\leq C\Vert f\Vert _{X},
\end{equation*}
for all $f\ $\ in the Schwartz space $\mathscr {S}(\mathbb{R}^{n})$, where
\begin{equation*}
T_{m}f=m(D)f=\mathscr {F}^{-1}(m\mathscr{F}f)
\end{equation*}
is the Fourier multiplier operator associated with $m$, $\scrF$ and $\scrF^{-1}$ denote the Fourier and inverse Fourier transform respectively,
and $m$ is called
the symbol or multiplier of $T_{m}$.

In particular, unimodular multipliers arise naturally
when one
solves the Cauchy problem for some dispersive equations.
For example, consider the Cauchy problem of dispersive equation
\begin{equation}
\begin{cases}
~i\partial_tu+(-\Delta )^{\frac{\beta }{2}}u=0 \\
~u(0,x)=u_{0}
\end{cases}
,
\end{equation}
$(t,x)\in \mathbb{R}\times \mathbb{R}^{n}$, the formal solution is given by $
u(t,x)=e^{it|D|^{\beta }}u_{0}(x)$, where $e^{it|D|^{\beta }}$ is just the
unimodular Fourier multiplier defined by $e^{it|D|^{\beta }}f=\scrF^{-1}(e^{it|\xi |^{\beta
}}\scrF f)$. Note that $e^{it|D|^{\beta }}u_{0}$ is the free solution of the Schr\"{o}dinger equation when $\b=2$,
and the free solution of the wave equation when $\b=1$.
Hence, in order to study the dispersive equation, it is of great interest to
study the boundedness of unimodular Fourier multipliers on several function spaces.

Taking $e^{i|D|^{\b}}$ as the prototype, the unimodular Fourier multiplier
in $\rn$ is defined by
\be
e^{i\mu(D)}f(x):= \int_{\rn}e^{2\pi ix\xi}e^{i\mu(\xi)}\hat{f}(\xi)d\xi,
\ee
where $\mu$ is a real-valued function (with some suitable assumptions).

Since $e^{i|\xi|^{\b}}\in L^{\fy}(\rn)$, the unimodular multiplier $e^{i|D|^{\b}}$
preserves the $L^2(\rn)$-norm. But in generally, $e^{i|D|^{\b}}$ is not bounded on $L^p(\rn)$ if $p\neq 2$.
Surprising, it was proved by B$\acute{e}$nyi-Gr\"{o}chenig-Okoudjou-Rogers \cite{Benyi2007JFA}
that $e^{i|D|^{\b}} (0<\b\leq 2)$ is
bounded on the modulation spaces $M_{p,q}^{s}$ for all $1\leq p,q\leq \infty $, $s\in \mathbb{R}$.
Furthermore, in the case $\beta >2$, Miyachi-Nicola-Rivetti-Tabacco-Tomita \cite{Miyachi2009PAMS}
showed that, for $1\leq p,q\leq \infty $ and $s\in
\mathbb{R}$, $e^{i|D|^{\beta }}$ is bounded from $M_{p,q}^{s}$ to $
M_{p,q}$ if and only if $s\geq (\beta -2)n|1/p-1/2|$. Roughly speaking, the boundedness behavior of $e^{i|D|^{\b}}$
is better on the modulation spaces than on the Lebesgue spaces.

Modulation space, introduced by Feichtinger \cite{Feichtinger1983}  in 1983, are defined by measuring
the size of a function or temperate distribution in the time-frequency plane
plane. Because of its better boundedness property of unimodular Fourier multipliers mentioned above, and its
better product and convolution properties (see \cite{GuoChenFanZhao2018MJM}), the modulation space have
been regarded as appropriate function spaces for the study of partial differential
equations (see \cite{RuzhanskySugimotoWang2012, WangHudzik2007JDE}). Concrete definition of modulation space will be given in Section 2.
For recent research on the boundedness result on modulation space of unimodular multipliers
one can also see e.g.\cite{CorderoNIcola2010JFAA, HuangChenFanZhu2018Banach, Tomita2010, ZhaoChenFanGuo2016NLAT}.

Among numerous references, we observe that in a recent work by Nicola-Primo-Tabacoo \cite{NicolaTabacoo2018JPDOA},
the authors use a more structured approach to deal with the boundedness of $e^{i\mu(D)}$.
In fact, inspired by Concetti-Toft \cite{ConcettiToft2009Arkiv}, Nicola-Primo-Tabacoo use the Taylor expansion of $e^x$,
to connect the regularity information between $\mu$ and $e^{i\mu}$.
Note that this Taylor expansion technique was also used in the proof of \cite[Theorem 9]{Benyi2007JFA}.
We also remark that
the process of regularity transformation from $\mu$ to $e^{i\mu}$  was in fact hidden in the action of
"taking the derivation of $e^{i\mu(\xi)}$" in some previous reference (see \cite{Miyachi2009PAMS, ZhaoChenFanGuo2016NLAT}).

In 1980s, H. G. Feichtinger \cite{Feichtinger1980Wiener} developed a
far-reaching generalization of amalgam spaces (which are now called Wiener amalgam spaces),
which allows a wide range of Banach spaces to serve as local or global components.
In this paper, we consider a special Wiener amalgam space, namely, the Wiener amalgam spaces $\wpqs$,
which can be viewed as the Fourier transform of modulation space.
Similar to modulation spaces, due to the localization property,
Wiener amalgam spaces also have
better boundedness results of unimodular Fourier multipliers and
better product and convolution properties (see \cite{GuoChenFanZhao2018MJM}).
Thus, it is also regarded to be an appropriate function spaces in studying PDE problems
(see \cite{CorderoNicola2007JFA,CorderoNicola2008JMAA, CorderoNicola2008JDE}).

In contrast to the fully development of research on the modulation spaces $\mpqs$,
the behaviors on the Wiener amalgma spacecs $\wpqs$ of most unimodular multipliers
still remains open on both sides of sufficiency and necessity.
As far as we know, there are only few results concerning this topic, one can find them in \cite{Cunanan2014JMAA, KatoTomita2017Arxiv}.
For the unimodular multiplier whose prototype is $e^{i|D|^{\b}}$,
the previous results associated with $\b\in (0,2]$ can be listed as follows.
We put the discussion about the high growth case $\b>2$ in Section 5.
\\\\
\textbf{Theorem A\ (\cite[Corollary 2.1]{Cunanan2014JMAA}).} Suppose $1\leq p,q\leq \fy$, $0<\b\leq 1$,
then $e^{i|D|^{\b}}$ is bounded on $\wpq$.
\\\\
\textbf{Theorem B\ (\cite[Proposition 3.1]{Cunanan2014JMAA}).} Suppose $1\leq p,q\leq \fy$,
 $0<\b\leq 2$. Let $\mu$ be a real-valued $C^{[n/2]+1}(\rn\bs\{0\})$ function satisfying
 \be
 |\partial^{\g}\mu(\xi)|\lesssim |\xi|^{\b-|\g|},\ \ \ |\g|\leq [n/2]+1.
 \ee
Then $e^{i\mu(D)}: W^{p,q}_{\d} \rightarrow \wpq$ is bounded for $\d\geq n\b/2|1/p-1/q|$ with the strict inequality when $p\neq q$.
\\~

Recently, Kato-Tomita\cite{KatoTomita2017Arxiv} give a sharp estimate of Schr\"{o}dinger operator $e^{i|D|^2}$.
\\
\textbf{Theorem C\ (\cite[Theorem 1.1]{KatoTomita2017Arxiv}).} Suppose $1\leq p,q\leq \fy$.
Then $e^{i|D|^2}: W^{p,q}_{\d} \rightarrow \wpq$ is bounded if and only if $\d\geq n|1/p-1/q|$ with the strict inequality when $p\neq q$.
\\~

We have two observations concerning Theorem A-C.

\bn
\item
Comparing with Theorem A, one can see that the conclusion in Theorem B is obviously not optimal.
Since $e^{i|D|}$ is bounded on $\wpq$ in Theorem A, but the corresponding boundedness result
in Theorem B with $\b=1$ need a potential loss greater than $n/2|1/p-1/q|$.
\item
Theorem C is a sharp estimate of boundedness result of Schr\"{o}dinger operator.
However, Schr\"{o}dinger operator is so special that the corresponding estimates can be obtained conveniently in a precise form.
This convenience will disappear even when considering the prototype $\mu(\xi)=|\xi|^{\b}$ for $1<\b<2$.
\en

Part of the present work can be regarded as a response to the two observations mentioned above.
We focus on the sufficiency and necessity for
the boundedness on the Wiener amalgam spaces $\wpqs$ in the full range $0<p,q\leq \fy, s\in \mathbb{R}$,
of $e^{i\mu(D)}$ whose prototype is $e^{i|D|^{\b}} (0<\b\leq 2)$.
First, we state our results for sufficiency. Denote by $\dot{p}:= \min\{1,p\}$,
$\dot{q}:= \min\{1,q\}$ and $\dot{p}\wedge \dot{q}=\min\{\dot{p},\dot{q}\}$.

\begin{theorem}[Potential persistence]\label{thm, wiener, potential persistence}
  Suppose $0<p,q\leq \infty$.
  Let $\mu\in C^1(\bbR^n)$ be a real-valued function satisfying
   \be
  \nabla \mu\in W^{\infty,\fy}_{n/(\dot{p}\wedge \dot{q})+\ep}
  \ee
  for some $\ep>0$.
  Then $e^{i\mu(D)}$ is bounded on $\wpq$.
\end{theorem}

\begin{theorem}[Potential loss]\label{thm, wiener, potential loss}
  Suppose $0<p,q\leq \infty$.
  Let $\mu$ be a real-valued $C^1(\bbR^n)$ function satisfying
  \be
  \langle \xi \rangle^{-s}\nabla \mu(\xi) \in (W^{\infty,\fy}_{n/(\dot{p}\wedge \dot{q})+\ep})^n
  \ee
  for some $s,\ep>0$.
  Then $e^{i\mu(D)}: W^{p,q}_{\d} \rightarrow \wpq$ is bounded for $\d>sn/(\dot{p}\wedge \dot{q})$.
\end{theorem}

\begin{theorem}[Interpolation case]\label{thm, interpolation case}
  Suppose $0<p,q\leq \infty$.
  Let $\mu$ be a real-valued $C^2(\bbR^n)$ function satisfying
  \be
  \langle \xi \rangle^{-s}\nabla \mu(\xi) \in (W^{\infty,\fy}_{n/(\dot{p}\wedge \dot{q})+\ep})^n,
  \ \
  \partial^{\g}\mu \in W^{\infty,\fy}_{n/\dot{p}+\ep}\ (|\g|=2),
  \ee
  for some $s,\ep>0$.
  Then $e^{i\mu(D)}: W^{p,q}_{\d} \rightarrow \wpq$ is bounded for $\d\geq sn|1/p-1/q|$
  with strict inequality when $p\neq q$.
\end{theorem}

\begin{corollary}[Derivative condition]\label{coy, bd, exact conditions}
  Suppose $0<p,q\leq \infty$.
  Let $\ep>n(1/\dot{p}-1)$, $\b\in (0,2]$.
  Let $\mu$ be a real-valued function of class $C^{[n/(\dot{p}\wedge \dot{q})]+3}$ on $\mathbb{R}^n \backslash \{0\}$ which satisfies
\begin{equation}
|\partial^{\gamma}\mu(\xi)|\leq C_{\gamma}|\xi|^{\epsilon-|\gamma|}, \hspace{5mm} 0<|\xi|\leq 1,~|\gamma|\leq[n/(1/\dot{p}-1/2)]+1,
\end{equation}
and
\begin{equation}
\begin{cases}
  |\partial^{\gamma}\mu(\xi)|\leq C_{\gamma}|\xi|^{\beta-1}, \hspace{5mm} |\xi|>1,~1\leq|\gamma|\leq[n/(\dot{p}\wedge \dot{q})]+2,\ \text{if}
  \ \b\in (0,1],
  \\
  |\partial^{\gamma}\mu(\xi)|\leq C_{\gamma}|\xi|^{\beta-|\gamma|}, \hspace{5mm} |\xi|>1,~1\leq|\gamma|\leq[n/(\dot{p}\wedge \dot{q})]+3,\ \text{if}
  \ \b\in (1,2].
\end{cases}
\end{equation}
Then $e^{i\mu(D)}: W^{p,q}_{\d} \rightarrow \wpq$ is bounded for $\d\geq n|1/p-1/q|\max\{\b-1,0\}$
  with strict inequality when $\b>1, p\neq q$.
\end{corollary}

Next, we give the following theorem for the necessity part.

\begin{theorem}[Sharp potential loss]\label{thm, wiener, sharp potential loss}
  Suppose $0<p,q\leq \infty$.
  Let $1<\b\leq 2$, and let $\mu$ be a real-valued $C^2(\bbR^n\bs \{0\})$ function
  satisfying
  \be
  (1-\r_0)\partial^{\g}\mu \in W^{\infty,\fy}_{n/(\dot{p}\wedge \dot{q})+\ep}\ (|\g|=2)\ \ \  \text{for some}\  \ep>0,
  \ee
  where $\r_{0}$  is a $C_c^{\fy}$ function supported in $B(0,1/2)$ and
  satisfies $\r_0(\xi)=1$ on $B(0,1/4)$.
  Suppose that the Hessian determinant of $\mu$ is not zero at some point $\k_0$ with $|\k_0|=1$.
  Moreover,
  \be
  \mu(\la\xi)=\la^{\b}\mu(\xi),\ \ \ \la\geq 1,\ \ \xi\in B(\k_0, r_0)\cap \bbS^{n-1}
  \ee
  for some $r_0>0$.
  Then the boundedness of $e^{i\mu(D)}: W^{p,q}_{s} \rightarrow \wpq$ implies
  \be
  s\geq n(\b-1)|1/p-1/q|
  \ee
  with strict inequality when $p\neq q$.
  \end{theorem}

  The following remark show that there exist plenty of functions satisfy the assumptions
  in Theorem \ref{thm, wiener, sharp potential loss}. The proof will be stated in the end of Section 4.
  \begin{remark}\label{rk, plenty funtions}
    Let $\b\in (1,2]$ and let $\mu$ be a real-valued homogeneous $C^{\fy}(\rn\bs \{0\})$ function of degree $\b$,
    satisfying $\mu(\xi)\neq 0$ for all $\xi\neq 0$.
    Then $\mu$ satisfies all the assumptions in Theorem \ref{thm, wiener, sharp potential loss}.
  \end{remark}

  As an application, we give the following characterization for the boundedness of our prototype $e^{i|D|^{\b}}$.
  \begin{corollary}[Return to the prototype]\label{coy, wiener, sharp potential loss}
  Let $1\leq p\leq \fy$, $0<q\leq\fy$,
  $\b\in (0,2]$. We have $e^{i|D|^{\b}}: W_{p,q}^{\d}\rightarrow \wpq$ is bounded if and only if
  \be
  \d\geq n|1/p-1/q|\max\{\b-1,0\}
  \ee
  with strict inequality when $\b>1, p\neq q$.
  \end{corollary}
Our new contribution are listed as follows.
\bn
\item For the sufficiency part, we first establish a useful criterion for the boundedness on the Wiener amalgam spaces in the full range $0<p,q\leq \fy$
of $e^{i\mu(D)}$ whose prototype is $e^{i|D|^{\b}}$ with $\b\in (0,2]$.
\item For the necessity part, we first determine the sharp potential loss for a large family of unimodular multipliers
on Wiener amalgam spaces.
\item As an application, we obtain the sharp $W_{p,q}^{\d}-\wpq$ boundedness results
for the operators $e^{i|D|^{\b}}$ with $\b\in (0,2]$.
Even in this special case, our results are an essential improvement of the previous results (see Theorem B and Theorem C).
\en

The rest of this paper is organized as follows.
In Section 2, we collected a number of definitions and auxiliary results, including
the dilation and convolution properties of certain Wiener amalgam spaces.

Section 3 is devoted to the proof of theorems and corollaries for the sufficiency part.
The first key point is to
find a suitable working space for $e^{i\mu}$, such that the functions in this working space can be localized in the time plane.
Although the convolution relations $W^{\dot{p}\wedge \dot{q},\fy}_{-\d}\ast W^{p,q}_{\d}\subset \wpq$ supply
a natural working space $M^{\fy,\dot{p}\wedge \dot{q}}=\scrF W^{\dot{p}\wedge \dot{q},\fy}$
for $\langle \cdot \rangle^{-\d}e^{i\mu}$ to live in, we observe that
$M^{\fy,\dot{p}\wedge \dot{q}}$ is not a suitable spaces for time localization.
In consideration of the time localization property, and the optimal embedding relation with $M^{\fy,\dot{p}\wedge \dot{q}}$,
we choose $W^{\fy,\fy}_{n/(\dot{p}\wedge \dot{q})+\ep}=M^{\fy,\fy}_{n/(\dot{p}\wedge \dot{q})+\ep}$ as our desirable working space.
Then, by time localization and Taylor expansion (of order 1), we give the proofs of Theorem \ref{thm, wiener, potential persistence}
and Theorem \ref{thm, wiener, potential loss}.
Combining with the boundedness property on the diagonal line case $p=q$,
we further give the proof of Theorem \ref{thm, interpolation case} by an interpolation argument.
Finally, Corollary \ref{coy, bd, exact conditions} is proved by establishing an embedding relations between
the working space and the function space with bounded derivative.

Section 4 contains the proof of Theorem \ref{thm, wiener, sharp potential loss} and Corollary \ref{coy, wiener, sharp potential loss}. Under the assumption of non-degenerate Hessian matrix, we find that the set
$\{\nabla \mu(\xi): \xi\in \rn\}$ is scattered in the sense that
the distance between
$\nabla \mu(\xi_1)$ and $\nabla \mu(\xi_2)$ has a lower bound according to the distance between $\xi_1$ and $\xi_2$, and
according to the value of $\b$.
In fact, the scattered degree of $\{\nabla \mu(\xi): \xi\in \rn\}$ is increasing as the value of $\b\in (1,2]$.
In particularly, when $\b=2$, the scattered degree of $\{\nabla \mu(\xi): \xi\in \rn\}$ is just like that of $\{\xi: \xi\in \rn\}$.
By this observation, we give a detailed scattered property in Lemma \ref{lemma, scattered},
which can be further used to obtain some useful estimates of special functions in Lemma \ref{lemma, estimates of special}.
Then, by a rotation trick, the boundedness of $e^{i\mu(D)}$ can be reduced to the simple embedding of
certain discrete sequences, which leads to our desired conclusion in Theorem \ref{thm, wiener, sharp potential loss}.
As a return to the prototype, Corollary \ref{coy, wiener, sharp potential loss} is the direct conclusion of
Corollary \ref{coy, bd, exact conditions} and Theorem \ref{thm, wiener, sharp potential loss}.

Some complements are prepared in the Section 5. In the high growth case of $\mu$
whose prototype is $\mu(\xi)=|\xi|^{\b}(\b>2)$, we find that
the working space should be replaced by some Wiener amalgam space without potential,
just like the working space used in the modulation case.
Based on this observation, we give some boundedness results for the high growth case.

\section{PRELIMINARIES}

We start this section by recalling some notations. Let $C$ be a positive
constant that may depend on $n,p,q,s,\beta .$ The
notation $X\lesssim Y$ denotes the statement that $X\leq CY$, the notation $%
X\sim Y$ means the statement $X\lesssim Y\lesssim X$.
For a multi-index $%
k=(k_{1},k_{2},...k_{n})\in \mathbb{Z}^{n}$, we denote%
\begin{equation*}
|k|_{\infty }:=\max_{i=1,2...n}|k_{i}|,\ \langle k\rangle
:=(1+|k|^{2})^{\frac{1}{2}}.
\end{equation*}

In this paper, for the sake of simplicity, we use the notation $"\mathscr{L}"$ to denote some large positive number which may be changed
corresponding to the exact environment.

Let $\mathscr {S}:= \mathscr {S}(\mathbb{R}^{n})$ be the Schwartz space
and $\mathscr {S}':=\mathscr {S}'(\mathbb{R}^{n})$ be the space of tempered distributions.
We define the Fourier transform $\mathscr {F}f$ and the inverse Fourier transform $\mathscr {F}^{-1}f$  of $f\in \mathscr {S}(\mathbb{R}^{n})$ by
$$
\mathscr {F}f(\xi)=\hat{f}(\xi)=\int_{\mathbb{R}^{n}}f(x)e^{-2\pi ix\cdot \xi}dx
~~
,
~~
\mathscr {F}^{-1}f(x)=\hat{f}(-x)=\int_{\mathbb{R}^{n}}f(\xi)e^{2\pi ix\cdot \xi}d\xi.
$$

We recall some definitions of the function spaces treated in this paper.
\begin{definition}
Let $0<p \leq \infty$, $s\in \mathbb{R}$. The weighted Lebesgue space $L_{x, s}^p$ consists of all measurable functions $f$ such that
\begin{numcases}{\|f\|_{L_{x, s}^p}:=}
     \left(\int_{\mathbb{R}^n}|f(x)|^p \langle x\rangle^{sp} dx\right)^{{1}/{p}}, &$p<\infty$   \\
     ess\sup_{x\in \mathbb{R}^n}|f(x)\langle x\rangle^s|,  &$p=\infty$
\end{numcases}
is finite.
If $f$ is defined on $\mathbb{Z}^n$, we denote
\begin{numcases}{\|f\|_{l_{k,s}^{p}}:=}
\left(\sum_{k\in \mathbb{Z}^n}|f(k)|^p \langle k\rangle^{sp}\right)^{{1}/{p}}, &$p<\infty$
\\
\sup_{k\in \mathbb{Z}^n}|f(k)\langle k\rangle^s|,\hspace{15mm} &$p=\infty$
\end{numcases}
and $l_{k,s}^p$ as the (quasi) Banach space of functions $f: \mathbb{Z}^n\rightarrow \mathbb{C}$ whose $l_{k,s}^p$ norm is finite.
We write $L_s^p$, $l_s^{p}$ for short, respectively, if there is no confusion.
\end{definition}

We recall an embedding relation between weighted sequences.
This lemma is easy to be verified, so we omit the proof here.

\begin{lemma}\label{lemma, discrete embedding}
Suppose $0<q_1,q_2\leq \infty$, $s_1,s_2\in \mathbb{R}$. Then
\be
l^{q_1}_{s_1}\subset l^{q_2}_{s_2}
\ee
holds if and only if one of the following holds:
\begin{eqnarray}
&&
\frac{1}{q_2}\leq \frac{1}{q_1},~s_2\leq s_1,
\\
&&
\frac{1}{q_2}> \frac{1}{q_1},~\frac{1}{q_2}+\frac{s_2}{n}< \frac{1}{q_1}+\frac{s_1}{n}.
\end{eqnarray}
\end{lemma}

Now, we turn to the definition of modulation and Wiener amalgam space.
Fixed a nonzero function $\phi\in \mathscr{S}$, the short-time Fourier
transform of $f\in \mathscr{S}'$ with respect to the window $\phi$ is given by
\begin{equation}
V_{\phi}f(x,\xi)=\langle f,M_{\xi}T_x\phi\rangle=\int_{\mathbb{R}^n}f(y)\overline{\phi(y-x)}e^{-2\pi iy\cdot \xi}dy,
\end{equation}
where the translation operator is defined as $T_{x_0}f(x):=f(x-x_0)$ and
the modulation operator is defined as $M_{\xi}f(x):=e^{2\pi i\xi \cdot x}f(x)$, for $x$, $x_0$, $\xi\in\mathbb{R}^n$.

\begin{definition}\label{Definition, modulation space, continuous form}
Let $0<p, q\leqslant \infty$, $s\in \mathbb{R}$.
Given a window function $\phi\in \mathscr{S}\backslash\{0\}$, the modulation space $\mathcal {M}^{p,q}_s$ consists
of all $f\in \mathscr{S}'(\mathbb{R}^n)$ such that the norm
\begin{equation}
\begin{split}
\|f\|_{\mathcal {M}^{p,q}_s}&=\big\|\|V_{\phi}f(x,\xi)\|_{L^p_{x}}\big\|_{L^q_{\xi,s}}
\\&
=\left(\int_{\mathbb{R}^n}\left(\int_{\mathbb{R}^n}|V_{\phi}f(x,\xi)|^{p} dx\right)^{{q}/{p}}\langle\xi\rangle^{sq}d\xi\right)^{{1}/{q}}
\end{split}
\end{equation}
is finite, with the usual modification when $p=\infty$ or $q=\infty$.
In addition, we write $\mathcal {M}^{p,q}:=\mathcal {M}^{p,q}_0$ for short.
\end{definition}

The above definition of \ $\mathcal {M}^{p,q}_s$ \  is independent of the choice of window function $\phi$.
One can see this fact in \cite{GrochenigBook2013} \ for the case $(p,q)\in\lbrack 1,\infty ]^{2}$,
and in \cite{GalperinSamarah2004ACHA} for the case $(p,q)\in (0,\infty ]^{2}\backslash\lbrack 1,\infty]^{2}$.
More properties of modulation spaces can be founded in \cite{Wang2011Book}.
One can also see \cite{Feichtinger2006} for a survey of modulation spaces.

Applying the frequency-uniform localization techniques, one can give an alternative definition of modulation spaces (see \cite{Triebel1983ZFA, WangHudzik2007JDE} for details).

Denote by $Q_{k}$ the unit cube with the center at $k$. Then the family $\{Q_{k}\}_{k\in\mathbb{Z}^{n}}$
constitutes a decomposition of $\mathbb{R}^{n}$.
Let $\rho \in \mathscr {S}(\mathbb{R}^{n}),$
$\rho: \mathbb{R}^{n} \rightarrow [0,1]$ be a smooth function satisfying that $\rho(\xi)=1$ for
$|\xi|_{\infty}\leq {1}/{2}$ and $\rho(\xi)=0$ for $|\xi|_{\fy}\geq 3/4$. Let
\begin{equation}
\rho_{k}(\xi)=\rho(\xi-k),  k\in \mathbb{Z}^{n}
\end{equation}
be a translation of \ $\rho$.
Since $\rho_{k}(\xi)=1$ in $Q_{k}$, we have that $\sum_{k\in\mathbb{Z}^{n}}\rho_{k}(\xi)\geq1$
for all $\xi\in\mathbb{R}^{n}$. Denote
\begin{equation}
\sigma_{k}(\xi):=\rho_{k}(\xi)\left(\sum_{l\in\mathbb{Z}^{n}}\rho_{l}(\xi)\right)^{-1},  ~~~~ k\in\mathbb{Z}^{n}.
\end{equation}
It is easy to know that $\{\sigma_{k}(\xi)\}_{k\in\mathbb{Z}^{n}}$
constitutes a smooth decomposition of $\mathbb{R}^{n}$, and $%
\sigma_{k}(\xi)=\sigma(\xi-k)$.
The frequency-uniform decomposition
operators can be defined by
\begin{equation}
\Box_{k}:= \mathscr{F}^{-1}\sigma_{k}\mathscr{F}
\end{equation}
for $k\in \mathbb{Z}^{n}$.
Now, we give the (discrete) definition of modulation space $\mpqs$.

\begin{definition}
Let $s\in \mathbb{R}, 0<p,q\leq \infty$. The modulation space $\mpqs$ consists of all $f\in \mathscr{S}'$ such that the (quasi-)norm
\begin{equation}
\|f\|_{\mpqs}:=\left( \sum_{k\in \mathbb{Z}^{n}}\langle k\rangle ^{sq}\|\Box_k f\|_{p}^{q}\right)^{1/q}
\end{equation}
is finite. We write $\mpq:=M^{p,q}_0$ for short.
\end{definition}

\begin{remark}
The above definition is independent of the choice of $\sigma$.
So, we can choose suitable $\sigma$ fitting our case.
In the definition above, the function sequence  $\{\sigma_k(\xi)\}_{k\in \mathbb{Z}^n}$ satisfies
$\sigma_k(\xi)=1$ and $\sigma_k(\xi)\sigma_l(\xi)=0$ when $|\xi-k|_{\infty}\leq 1/4$ and $k\neq l$.
We also recall that the definition of $\mathcal {M}^{p,q}_s$ and $M^{p,q}_s$ are equivalent.
In this paper, we mainly use $M^{p,q}_s$ to denote the modulation space.
\end{remark}

\begin{definition}\label{Definition, Wiener amalgam space, continuous form}
Let $0<p, q\leqslant \infty$, $s\in \mathbb{R}$.
Given a window function $\phi\in \mathscr{S}\backslash\{0\}$, the Wiener amalgam space $W^{p,q}_s$ consists
of all $f\in \mathscr{S}'(\mathbb{R}^n)$ such that the norm
\begin{equation}
\begin{split}
\|f\|_{W^{p,q}_{s}}&=\big\|\|V_{\phi}f(x,\xi)\|_{L^q_{\xi, s}}\big\|_{L^p_{x}}
\\&
=\left(\int_{\mathbb{R}^n}\left(\int_{\mathbb{R}^n}|V_{\phi}f(x,\xi)|^{q}\langle \xi\rangle^{sq}d\xi\right)^{{p}/{q}}dx\right)^{{1}/{p}}
\end{split}
\end{equation}
is finite, with the usual modifications when $p=\infty$ or $q=\infty$.
\end{definition}

\begin{lemma}[Equivalent norm of $\wpqs$, \cite{GuoWuYangZhao2017JFA}]\label{lemma, equi norm of Wiener}
  Let $0<p, q\leqslant \infty$, $s\in \mathbb{R}$. Then
  \be
  \|f\|_{\wpqs}\sim
  \left(\sum_{k\in \zn}\|\s_kf\|^p_{W^{p,q}_s}\right)^{1/p}
  \sim
  \left(\sum_{k\in \zn}\|\s_kf\|^p_{\scrF L^q_s}\right)^{1/p}
  \ee
  with usual modification when $p=\fy$.

\end{lemma}

By the definition of modulation and Wiener amalgam spaces and the fact $|V_{\phi }f(x,\xi )|=|V_{\hat{\phi}}\hat{f}(\xi ,-x)|$,
we immediately have
\be
M^{p,p}_s=W^{p,p}_s,\ \ \ \scrF M^{p,q}=W^{q,p}, \ \ \scrF W^{p,q}=M^{q,p}.
\ee

Following we collect some useful properties on Wiener amalgam spaces.
The first one is the convolution and product relations.
We put the proof in Appendix A.
\begin{lemma}[Convolution relations]\label{lemma, conv of Wfy}
Let $p,q\in (0,\fy]$, $s> n$, $\d\in \mathbb{R}$, we have
\bn
\item $W^{p,q}_{\d}\ast W^{\dot{p}\wedge \dot{q},\fy}_{-\d}\subset W^{p,q}$;
\item $W^{\fy,\fy}_{s}\cdot W^{\fy,\fy}_{s}\subset W^{\fy,\fy}_{s}$.
\en
\end{lemma}

\begin{lemma}[Dilation property of $W^{\fy,\fy}_{s}$, \cite{WangHan2014JMSP} Theorem 3.3]\label{lemma, dila of Wfy}
Let $p\in (0,\fy]$, $s> n$  we have
\be
\|f_t\|_{W^{\fy,\fy}_{s}}\lesssim \|f\|_{W^{\fy,\fy}_{s}},
\ee
 where $f_t(x):=f(tx)$, $t\in (0,1]$.
\end{lemma}
\begin{proof}
  For the self-containing of this paper, we give a short proof here.
  Write
  \be
  \|f_t\|_{W^{\fy,\fy}_{s}}=\|f_t\|_{M^{\fy,\fy}_{s}}
  =\sup_{k\in \zn}\lan k\ran^{s}\|\Box_kf_t\|_{L^{\fy}}.
  \ee
  Denote by
  \be
  A_{k,t}:=\{l\in \zn: \s_l(\xi)\s_k(t\xi)\neq 0\}.
  \ee
  Then
  \be
  \begin{split}
    \|\Box_kf_t\|_{L^{\fy}}
    = &
    \|\frac{1}{t^n}\scrF^{-1}(\s_k(\xi)\hat{f}(\xi/t))\|_{L^{\fy}}
    \\
    = &
    \|\scrF^{-1}(\s_k(t\xi)\hat{f}(\xi))\|_{L^{\fy}}
    \\
    = &
    \|\scrF^{-1}(\sum_{l\in A_{k,t}}\s_l(\xi)\s_k(t\xi)\hat{f}(\xi))\|_{L^{\fy}}
\end{split}
\ee
For $k$ near the origin, we have
\be
\begin{split}
   \lan k\ran^{s}\|\Box_k f_t\|_{L^{\fy}}
   \sim &
   \|\Box_k f_t\|_{L^{\fy}}
   \\
    \lesssim &
    \sum_{l\in A_{k,t}}\|\scrF^{-1}(\s_l(\xi)\hat{f}(\xi))\|_{L^{\fy}}\cdot \|\scrF^{-1}(\s_k(t\cdot))\|_{L^1}
    \\
    \lesssim &
    \sum_{l\in \zn}\lan l\ran^{-s}\lan l\ran^{s}\|\Box_lf\|_{L^{\fy}}
    \lesssim
    \sum_{l\in \zn}\lan l\ran^{-s}\|f\|_{W^{\fy,\fy}_{s}}\lesssim \|f\|_{W^{\fy,\fy}_{s}}.
  \end{split}
  \ee
For $k$ far away for the origin, observe that $|A_{k,t}|\sim t^{-n}$,
and
$|l| \sim |k/t|$ for $l\in A_{k,t}$.
We have
\be
\begin{split}
   \|\Box_k f_t\|_{L^{\fy}}
    \lesssim &
    |A_{k,t}|\sup_{l\in A_{k,t}}\|\scrF^{-1}(\s_l(\xi)f(\xi))\|_{L^{\fy}}
    \\
    \lesssim &
    t^{-n}\sup_{l\in A_{k,t}}\lan l\ran^{-s}\lan l\ran^{s}\|\Box_lf\|_{L^{\fy}}
    \\
    \lesssim &
    t^{-n}\lan k/t\ran^{-s}\|f\|_{W^{\fy,\fy}_{s}}
    \lesssim \lan k\ran^{-s}\|f\|_{W^{\fy,\fy}_{s}}.
  \end{split}
  \ee
The above two estimates yield that
\be
\begin{split}
  \|f_t\|_{W^{\fy,\fy}_{s}}
  \sim
  \sup_{k\in \zn}\lan k\ran^{s}\|\Box_kf_t\|_{L^{\fy}}
  \lesssim
  \|f\|_{W^{\fy,\fy}_{s}}.
\end{split}
\ee
\end{proof}

\begin{lemma}[Rotation free]\label{lemma, rotation free}
  Let $p,q\in (0,\fy]$, $s\in \mathbb{R}$, $P$ be an orthogonal matrix.
  Denote by $f_P(x):=f(P^{-1}x)$. We have
  \be
  \|f_P\|_{W^{p,q}_s}\sim\|f\|_{W^{p,q}_s}.
  \ee
\end{lemma}
\begin{proof}
Without loss of generality, we assume that $\phi$ is a radial real function.
  By a direct calculation we get
  \be
  \begin{split}
    V_{\phi}f_P(x,\xi)
    = &
    \int_{\rn}\widehat{f}(P^{-1}\eta)\widehat{\phi}(\eta-\xi)e^{2\pi ix\eta}d\eta
    \\
    = &
    \int_{\rn}\widehat{f}(\eta)\widehat{\phi}(P\eta-\xi)e^{2\pi ix\cdot P\eta}d\eta
    \\
    = &
    \int_{\rn}\widehat{f}(\eta)\widehat{\phi}(\eta-P^{-1}\xi)e^{2\pi iP^{-1}x\cdot \eta}d\eta
    \\
    = &
    V_{\phi}f(P^{-1}x,P^{-1}\xi).
  \end{split}
  \ee
  It follows that
  \be
  \begin{split}
    \|f_P\|_{\wpqs}
    = &
    \left\|\left\|V_{\phi}f_P(x,\xi)\right\|_{L_{\xi,s}^{q}}\right\|_{L^p_x}
    \\
    = &
    \left\|\left\|V_{\phi}f(P^{-1}x,P^{-1}\xi)\right\|_{L_{\xi,s}^{q}}\right\|_{L^p_x}
        \\
    = &
    \left\|\left\|V_{\phi}f(x,\xi)\right\|_{L_{\xi,s}^{q}}\right\|_{L^p_x}
    \sim
    \|f\|_{\wpqs}.
  \end{split}
  \ee
\end{proof}

Finally, we recall a disjoint property for the $\al$-covering, which will be used in the proof Lemma \ref{lemma, scattered}.
\begin{lemma}[See Lemma A.1 in \cite{GuoFanZhao2018SCIChina}]\label{lemma, position}
  Let $\alpha\in [0,1)$ be arbitrary. Then there is some $\delta>0$ such that the family
  \begin{equation*}
    \{B_k^{\delta}\}_{k\in \mathbb{Z}^n}:= \left\{B(\langle k\rangle^{\frac{\alpha}{1-\alpha}}k, \delta \langle k\rangle^{\frac{\alpha}{1-\alpha}})\right\}_{k\in \mathbb{Z}^n}
  \end{equation*}
  is pairwise disjoint.
\end{lemma}

\section{boundedness results}
In this section, we establish some boundedness results of the uimodular Fourier multiplier on Wiener amalgam spaces,
including the proof of potential persistence case Theorem \ref{thm, wiener, potential persistence},
and the proof of potential loss case Theorem \ref{thm, wiener, potential loss} and \ref{thm, interpolation case}.
Then, by establishing an embedding Lemma \ref{lemma, derivative to wiener}
related to the pointwise derivative conditions,
we give the proof of Corollary \ref{coy, bd, exact conditions}.

\subsection{Potential persistence: low growth of $\mu$}
This subsection is devoted to the proof of Theorem \ref{thm, wiener, potential persistence}.
Thanks to the time localization property of working space $W^{\dot{p}\wedge \dot{q},\fy}$,
the information of $\mu$ and $\nabla\mu$ can be related locally.

\begin{proof}[Proof of Theorems \ref{thm, wiener, potential persistence}]
By the convolution relation (see Lemma \ref{lemma, conv of Wfy})
\be
W^{p,q}\ast W^{\dot{p}\wedge \dot{q},\fy}\subset W^{p,q},
\ee
we only need to verify that $\scrF^{-1}e^{i\mu}\in W^{\dot{p}\wedge \dot{q},\fy}$, or equivalently,
\ben\label{pf, 1,1}
e^{i\mu}\in M^{\fy,\dot{p}\wedge \dot{q}}.
\een
Here, by the assumption of $\nabla \mu\in (W^{\infty,\fy}_{n/(\dot{p}\wedge \dot{q})+\ep})^n$, we will show that
\be
e^{i\mu}\in W^{\fy,\fy}_{n/(\dot{p}\wedge \dot{q})+\ep}.
\ee
Then \eqref{pf, 1,1} follows by the fact $W^{\fy,\fy}_{n/(\dot{p}\wedge \dot{q})+\ep}=M^{\fy,\fy}_{n/(\dot{p}\wedge \dot{q})+\ep}\subset M^{\fy,\dot{p}\wedge \dot{q}}$.
By an equivalent norm of Wiener amalgam space (see Lemma \ref{lemma, equi norm of Wiener}) we have
\be
\|e^{i\mu}\|_{W^{\fy,\fy}_{n/(\dot{p}\wedge \dot{q})+\ep}}\sim \sup_{k\in \bbZ^n}\|\s_ke^{i\mu}\|_{W^{\fy,\fy}_{n/(\dot{p}\wedge \dot{q})+\ep}}.
\ee
The Taylor's formula yields that
\be
\mu(k+\xi)=\mu(k)+\int_0^1 \xi\cdot\nabla \mu(k+t\xi)dt=:\mu(k)+R_k.
\ee
Then,
\be
\|\s_ke^{i\mu}\|_{W^{\fy,\fy}_{n/(\dot{p}\wedge \dot{q})+\ep}}=\|\s_0e^{i\mu(k+\cdot)}\|_{W^{\fy,\fy}_{n/(\dot{p}\wedge \dot{q})+\ep}}
=\|\s_0e^{iR_k}\|_{W^{\fy,\fy}_{n/(\dot{p}\wedge \dot{q})+\ep}}=\|\s_0e^{i\s_0^*R_k}\|_{W^{\fy,\fy}_{n/(\dot{p}\wedge \dot{q})+\ep}},
\ee
where $\s_k^{\ast}:=\sum_{\s_l\s_k\neq 0}\s_l.$
By the algebraic property $W^{\fy,\fy}_{n/(\dot{p}\wedge \dot{q})+\ep}\cdot W^{\fy,\fy}_{n/(\dot{p}\wedge \dot{q})+\ep}\subset W^{\fy,\fy}_{n/(\dot{p}\wedge \dot{q})+\ep}$ (see Lemma \ref{lemma, conv of Wfy}),
there exists a constant $C\geq 1$ such that
\be
\|f\cdot f\|_{W^{\fy,\fy}_{n/(\dot{p}\wedge \dot{q})+\ep}}\leq C \|f\|_{W^{\fy,\fy}_{n/(\dot{p}\wedge \dot{q})+\ep}}^2.
\ee
This implies that
\be
\begin{split}
  &\|\s_0e^{i\s_0^*R_k}\|_{W^{\fy,\fy}_{n/(\dot{p}\wedge \dot{q})+\ep}}
  =\|\s_0+\s_0(e^{i\s_0^*R_k}-1)\|_{W^{\fy,\fy}_{n/(\dot{p}\wedge \dot{q})+\ep}}
  \\
  \lesssim &
  1+\|\s_0\|_{W^{\fy,\fy}_{n/(\dot{p}\wedge \dot{q})+\ep}}\|e^{i\s_0^*R_k}-1\|_{W^{\fy,\fy}_{n/(\dot{p}\wedge \dot{q})+\ep}}
  \\
  \lesssim &
  1+\|e^{i\s_0^*R_k}-1\|_{W^{\fy,\fy}_{n/(\dot{p}\wedge \dot{q})+\ep}}
  =
  1+\left\|\sum_{m=1}^{\fy}\frac{(i\s_0^*R_k)^m}{m!}\right\|_{W^{\fy,\fy}_{n/(\dot{p}\wedge \dot{q})+\ep}}
  \\
  \leq &
  1+\sum_{m=1}^{\fy}\frac{\left\|(\s_0^*R_k)^m\right\|_{W^{\fy,\fy}_{n/(\dot{p}\wedge \dot{q})+\ep}}}{m!}
  \leq
  \sum_{m=0}^{\fy}\frac{C^m\left\|\s_0^*R_k\right\|_{W^{\fy,\fy}_{n/(\dot{p}\wedge \dot{q})+\ep}}^m}{m!}
  \leq
  \exp\left(C\left\|\s_0^*R_k\right\|_{W^{\fy,\fy}_{n/(\dot{p}\wedge \dot{q})+\ep}}\right).
\end{split}
\ee
Finally,
\be
\begin{split}
  \|\s_0^*R_k\|_{W^{\fy,\fy}_{n/(\dot{p}\wedge \dot{q})+\ep}}
  = &
  \left\|\s_0^*\int_0^1 \xi\cdot \nabla \mu(k+t\xi)dt\right\|_{W^{\fy,\fy}_{n/(\dot{p}\wedge \dot{q})+\ep}}
  \\
  \lesssim &
  \sum_{|\g|=1}\left\|\s_0^*\xi^{\g}\right\|_{W^{\fy,\fy}_{n/(\dot{p}\wedge \dot{q})+\ep}}\cdot
  \left\|\int_0^1 \partial^{\g} \mu(k+t\xi)dt\right\|_{W^{\fy,\fy}_{n/(\dot{p}\wedge \dot{q})+\ep}}
  \\
  \lesssim &
  \sum_{|\g|=1}
  \left\|\int_0^1 \partial^{\g} \mu(k+t\xi)dt\right\|_{W^{\fy,\fy}_{n/(\dot{p}\wedge \dot{q})+\ep}}
  \\
  \lesssim &
  \sum_{|\g|=1}
  \int_0^1 \left\|\partial^{\g} \mu(k+t\xi)\right\|_{W^{\fy,\fy}_{n/(\dot{p}\wedge \dot{q})+\ep}}dt
  \\
  =&
  \sum_{|\g|=1}\int_0^1 \left\|\partial^{\g} \mu(t\xi)\right\|_{W^{\fy,\fy}_{n/(\dot{p}\wedge \dot{q})+\ep}}dt
  \leq
  \sum_{|\g|=1}\left\|\partial^{\g} \mu\right\|_{W^{\fy,\fy}_{n/(\dot{p}\wedge \dot{q})+\ep}}£¬
\end{split}
\ee
where in the last inequality we use the dilation property in Lemma \ref{lemma, dila of Wfy}.

Combining the above estimates yields that
\be
\begin{split}
\|e^{i\mu}\|_{W^{\fy,\fy}_{n/(\dot{p}\wedge \dot{q})+\ep}}
\sim \sup_{k\in \bbZ^n}\|\s_ke^{i\mu}\|_{W^{\fy,\fy}_{n/(\dot{p}\wedge \dot{q})+\ep}}
\lesssim &
\sup_{k\in \bbZ^n}\exp\left(C\left\|\s_0^*R_k\right\|_{W^{\fy,\fy}_{n/(\dot{p}\wedge \dot{q})+\ep}}\right)
\\
\lesssim &
\exp\left(C\sum_{|\g|=1}\left\|\partial^{\g} \mu\right\|_{W^{\fy,\fy}_{n/(\dot{p}\wedge \dot{q})+\ep}}\right)\lesssim 1.
\end{split}
\ee
We have now completed this proof.
\end{proof}

\subsection{Potential loss: mild growth of $\mu$}
This subsection is devoted to the proof of Theorem \ref{thm, wiener, potential loss}.
In this case,
we apply some dilation argument to avoid the regularity loss on the exponent of $e$.

\begin{proof}[Proof of Theorems \ref{thm, wiener, potential loss}]
  Without loss of generality, we assume
  \ben\label{pf, 2,2}
  s\ep+sn/(\dot{p}\wedge \dot{q})-\d<0
  \een
  for fixed $s>0$ and $\d>sn/(\dot{p}\wedge \dot{q})$.
  Denote by
  \be
  P_{t}(x):=\langle x\rangle^{t},\ \ x\in \rn,\ t\in \bbR.
  \ee
  By the convolution relations $W^{p,q}_{\d}\ast W^{\dot{p}\wedge \dot{q},\fy}_{-\d}\subset W^{p,q}$, we only need to verify that
  \be
  \scrF^{-1}e^{i\mu}\in W^{\dot{p}\wedge \dot{q},\fy}_{-\d}.
  \ee
  Note that $W^{\fy,\fy}_{n/(\dot{p}\wedge \dot{q})+\ep}= M^{\fy,\fy}_{n/(\dot{p}\wedge \dot{q})+\ep}
  \subset M^{\fy,\dot{p}\wedge \dot{q}}$, if we have
  \ben\label{pf, 2,1}
  P_{-\d}e^{i\mu}\in W^{\fy,\fy}_{n/(\dot{p}\wedge \dot{q})+\ep},
  \een
  the desired conclusion follows by
  \be
  \begin{split}
  P_{-\d}e^{i\mu}\in W^{\fy,\fy}_{n/(\dot{p}\wedge \dot{q})+\ep}\Longrightarrow P_{-\d}e^{i\mu}\in M^{\fy,\dot{p}\wedge \dot{q}}
  \Longrightarrow & \scrF^{-1}(P_{-\d}e^{i\mu})\in W^{\dot{p}\wedge \dot{q},\fy}
  \\
  \Longrightarrow & \scrF^{-1}e^{i\mu}\in W^{\dot{p}\wedge \dot{q},\fy}_{-\d}.
  \end{split}
  \ee
  Write
  \be
  \|P_{-\d}e^{i\mu}\|_{W^{\fy,\fy}_{n/(\dot{p}\wedge \dot{q})+\ep}}
  \sim\sup_{k\in \bbZ^n}\|\s_kP_{-\d}e^{i\mu}\|_{W^{\fy,\fy}_{n/(\dot{p}\wedge \dot{q})+\ep}}
  \sim
  \sup_{k\in \bbZ^n}\langle k\rangle^{-\d}\|\s_ke^{i\mu}\|_{W^{\fy,\fy}_{n/(\dot{p}\wedge \dot{q})+\ep}}.
  \ee

  Denote by
  \be
  A_k=\{l\in \bbZ^n:\   \s_l\cdot \s_k(\frac{\cdot}{\langle k\rangle^s})\neq 0\}.
  \ee
  Then
  \be
  \begin{split}
    \|\s_ke^{i\mu}\|_{W^{\fy,\fy}_{n/(\dot{p}\wedge \dot{q})+\ep}}
    = &
    \|\s_k\sum_{l\in A_k}\s_l(\langle k\rangle^s\cdot)e^{i\mu}\|_{W^{\fy,\fy}_{n/(\dot{p}\wedge \dot{q})+\ep}}
    \\
    \lesssim &
    \|\s_k\|_{W^{\fy,\fy}_{n/(\dot{p}\wedge \dot{q})+\ep}}
    \|\sum_{l\in A_k}\s_l(\langle k\rangle^s\cdot)e^{i\mu}\|_{W^{\fy,\fy}_{n/(\dot{p}\wedge \dot{q})+\ep}}
    \\
    \lesssim &
    \|\sum_{l\in A_k}\s_l(\langle k\rangle^s\cdot)e^{i\mu}\|_{W^{\fy,\fy}_{n/(\dot{p}\wedge \dot{q})+\ep}},
  \end{split}
  \ee
  where we use the fact
  \be
  \|\s_k\|_{W^{\fy,\fy}_{n/(\dot{p}\wedge \dot{q})+\ep}}=\|\s_0\|_{W^{\fy,\fy}_{n/(\dot{p}\wedge \dot{q})+\ep}}\lesssim 1.
  \ee
  Denote
    \be
  \mu_k(x):=\mu(\frac{x}{\langle k\rangle^s}).
  \ee
  Let us turn to the estimate of
  $\|\sum_{l\in A_k}\s_l(\langle k\rangle^s\cdot)e^{i\mu}\|_{W^{\fy,\fy}_{n/(\dot{p}\wedge \dot{q})+\ep}}$ as follows:
  \be
  \begin{split}
    \|\sum_{l\in A_k}\s_l(\langle k\rangle^s\cdot)e^{i\mu}\|_{W^{\fy,\fy}_{n/(\dot{p}\wedge \dot{q})+\ep}}
    = &
    \|\sum_{l\in A_k}\s_l(\langle k\rangle^s\cdot)e^{i\mu}\|_{\scrF L^{\fy}_{n/(\dot{p}\wedge \dot{q})+\ep}}
    \\
    = &
    \|\scrF^{-1}(\sum_{l\in A_k}\s_l(\langle k\rangle^s\cdot)e^{i\mu})(x)\langle x\rangle^{n/(\dot{p}\wedge \dot{q})+\ep}\|_{L^{\fy}}
    \\
    = &
    \|\langle k\rangle^{-sn}\scrF^{-1}(\sum_{l\in A_k}\s_le^{i\mu_k})(\langle k\rangle^{-s}x)\langle x\rangle^{n/(\dot{p}\wedge \dot{q})+\ep}\|_{L^{\fy}}
    \\
    = &
    \|\langle k\rangle^{-sn}\scrF^{-1}(\sum_{l\in A_k}\s_le^{i\mu_k})(x)\langle \langle k\rangle^{s}x\rangle^{n/(\dot{p}\wedge \dot{q})+\ep}\|_{L^{\fy}}
    \\
    \lesssim &
    \langle k\rangle^{-sn}\cdot \langle k\rangle^{s(n/(\dot{p}\wedge \dot{q})+\ep)}
    \|\scrF^{-1}(\sum_{l\in A_k}\s_le^{i\mu_k})(x)\langle x\rangle^{n/(\dot{p}\wedge \dot{q})+\ep}\|_{L^{\fy}}
    \\
    \lesssim &
    \langle k\ran^{sn(1/(\dot{p}\wedge \dot{q})-1)}
    \langle k\rangle^{s\ep}
    \|\scrF^{-1}(\sum_{l\in A_k}\s_le^{i\mu_k})(x)\langle x\rangle^{n/(\dot{p}\wedge \dot{q})+\ep}\|_{L^{\fy}}
    \\
    \lesssim &
    \langle k\ran^{sn(1/(\dot{p}\wedge \dot{q})-1)}
    \langle k\rangle^{s\ep}
    \sum_{l\in A_k}\|\scrF^{-1}(\s_le^{i\mu_k})(x)\langle x\rangle^{n/(\dot{p}\wedge \dot{q})+\ep}\|_{L^{\fy}}.
  \end{split}
  \ee
  Observe that
  \be
  |A_k|\lesssim \langle k\rangle^{sn}.
  \ee
  We further have
  \be
  \begin{split}
  &\|\sum_{l\in A_k}\s_l(\langle k\rangle^s\cdot)e^{i\mu}\|_{W^{\fy,\fy}_{n/(\dot{p}\wedge \dot{q})+\ep}}
  \\
  \lesssim &
  \langle k\ran^{sn(1/(\dot{p}\wedge \dot{q})-1)}
  \langle k\rangle^{s\ep}
    \sum_{l\in A_k}\|\scrF^{-1}(\s_le^{i\mu_k})(x)\langle x\rangle^{n/(\dot{p}\wedge \dot{q})+\ep}\|_{L^{\fy}}
    \\
    \lesssim &
    \langle k\ran^{sn(1/(\dot{p}\wedge \dot{q})-1)}
    \langle k\rangle^{s\ep+sn}\sup_{l\in A_k}\|\scrF^{-1}(\s_le^{i\mu_k})(x)\langle x\rangle^{n/(\dot{p}\wedge \dot{q})+\ep}\|_{L^{\fy}}
    \\
    = &
    \langle k\rangle^{s(\ep+n/(\dot{p}\wedge \dot{q}))}\sup_{l\in A_k}\|\s_le^{i\mu_k}\|_{W^{\fy,\fy}_{n/(\dot{p}\wedge \dot{q})+\ep}}
    =
    \langle k\rangle^{s(\ep+n/(\dot{p}\wedge \dot{q}))}\sup_{l\in A_k}\|\s_0e^{i\mu_k(\cdot+l)}\|_{W^{\fy,\fy}_{n/(\dot{p}\wedge \dot{q})+\ep}}.
  \end{split}
  \ee

  Write
  \be
  \mu_k(\xi+l)=\mu_k(l)+\int_0^1 \xi\cdot\nabla \mu_k(l+t\xi)dt=:\mu_k(l)+R_l^k.
  \ee
  Then
  \be
  \begin{split}
    \|\s_0e^{i\mu_k(\cdot+l)}\|_{W^{\fy,\fy}_{n/(\dot{p}\wedge \dot{q})+\ep}}
    = &
    \|\s_0e^{iR_l^k}\|_{W^{\fy,\fy}_{n/(\dot{p}\wedge \dot{q})+\ep}}
    \\
    = &
    \|\s_0e^{i\s_0^*R_l^k}\|_{W^{\fy,\fy}_{n/(\dot{p}\wedge \dot{q})+\ep}}
    \\
    \lesssim &
    \|\s_0\|_{W^{\fy,\fy}_{n/(\dot{p}\wedge \dot{q})+\ep}}\exp\left( C\|\s_0^*R_l^k\|_{W^{\fy,\fy}_{n/(\dot{p}\wedge \dot{q})+\ep}}\right)
    \lesssim \exp\left( C\|\s_0^*R_l^k\|_{W^{\fy,\fy}_{n/(\dot{p}\wedge \dot{q})+\ep}}\right).
  \end{split}
  \ee
  Take $h$ to be a $C_c^{\fy}(\rn)$ function such that $h(\xi)=1$ on the support of $\s_0^*$.
  We have
  \be
  \begin{split}
    \|\s_0^*R_l^k\|_{W^{\fy,\fy}_{n/(\dot{p}\wedge \dot{q})+\ep}}
    = &
    \|\s_0^* \int_0^1 \xi\cdot\nabla \mu_k(l+t\xi)h(t\xi) dt\|_{W^{\fy,\fy}_{n/(\dot{p}\wedge \dot{q})+\ep}}
    \\
    \lesssim &
    \sum_{|\g|=1}\|\s_0^*\xi^{\g} \|_{W^{\fy,\fy}_{n/(\dot{p}\wedge \dot{q})+\ep}}\cdot
    \|\int_0^1 h(t\xi)\partial^{\g} \mu_k(l+t\xi)dt\|_{W^{\fy,\fy}_{n/(\dot{p}\wedge \dot{q})+\ep}}
    \\
    \lesssim &
\sum_{|\g|=1}\int_0^1 \|h(t\xi)\partial^{\g} \mu_k(l+t\xi)\|_{W^{\fy,\fy}_{n/(\dot{p}\wedge \dot{q})+\ep}}dt
\\
     \lesssim &
     \sum_{|\g|=1}
     \|h(\xi)\partial^{\g} \mu_k(l+\xi)\|_{W^{\fy,\fy}_{n/(\dot{p}\wedge \dot{q})+\ep}}
     =
     \sum_{|\g|=1}
     \|h(\cdot-l)\partial^{\g} \mu_k\|_{W^{\fy,\fy}_{n/(\dot{p}\wedge \dot{q})+\ep}}.
  \end{split}
  \ee

  To estimate $\|h(\cdot-l)\partial^{\g} \mu_k\|_{W^{\fy,\fy}_{n/(\dot{p}\wedge \dot{q})+\ep}}$,
  we will use the assumption that $\langle \xi \rangle^{-s}\nabla \mu(\xi) \in (W^{\infty,\fy}_{n/(\dot{p}\wedge \dot{q})+\ep})^n$.
  Choose a $C_c^{\fy}(\rn)$ function $g$ such that
  \be
  g(\frac{\xi}{\lan k\ran^s}-k)h(\xi-l)=h(\xi-l),\ \ \ l\in A_k.
  \ee
  Then for $|\g|=1$ we have
  \be
  \begin{split}
  \langle k\rangle^{-s}\|g(\cdot-k)\partial^{\g} \mu\|_{W^{\infty,\fy}_{n/(\dot{p}\wedge \dot{q})+\ep}}
  \sim &
  \|g(\cdot-k)P_{-s}\partial^{\g} \mu\|_{W^{\infty,\fy}_{n/(\dot{p}\wedge \dot{q})+\ep}}
  \\
  \lesssim &
  \|g(\cdot-k)\|_{\wyynp}\|P_{-s}\partial^{\g} \mu\|_{W^{\infty,\fy}_{n/(\dot{p}\wedge \dot{q})+\ep}}
  \\
  \lesssim &
  \|P_{-s}\partial^{\g} \mu\|_{W^{\infty,\fy}_{n/(\dot{p}\wedge \dot{q})+\ep}}\lesssim 1.
  \end{split}
  \ee
  By dilation property of $W^{\infty,\fy}_{n/(\dot{p}\wedge \dot{q})+\ep}$, we get
  \be
  \|g(\frac{\cdot}{\lan k\ran^s}-k)\partial^{\g} \mu(\frac{\cdot}{\langle k\rangle^s})\|_{W^{\infty,\fy}_{n/(\dot{p}\wedge \dot{q})+\ep}}
  \lesssim
  \|g(\cdot-k)\partial^{\g} \mu\|_{W^{\infty,\fy}_{n/(\dot{p}\wedge \dot{q})+\ep}}.
  \ee
  Then
  \be
  \begin{split}
  \|g(\frac{\cdot}{\lan k\ran^s}-k)(\partial^{\g} \mu_k)(\cdot)\|_{W^{\infty,\fy}_{n/(\dot{p}\wedge \dot{q})+\ep}}
  = &
  \langle k\rangle^{-s}\|g(\frac{\cdot}{\lan k\ran^s}-k)\partial^{\g} \mu(\frac{\cdot}{\langle k\rangle^s})\|_{W^{\infty,\fy}_{n/(\dot{p}\wedge \dot{q})+\ep}}
  \\
  \lesssim &
  \langle k\rangle^{-s}\|g(\cdot-k)\partial^{\g} \mu\|_{W^{\infty,\fy}_{n/(\dot{p}\wedge \dot{q})+\ep}}
  \\
  \lesssim &
  \|P_{-s}\partial^{\g} \mu\|_{W^{\infty,\fy}_{n/(\dot{p}\wedge \dot{q})+\ep}}\lesssim 1.
  \end{split}
  \ee
  For every $l\in A_k$, we get
  \be
  \begin{split}
    \|h(\cdot-l)\partial^{\g} \mu_k\|_{W^{\fy,\fy}_{n/(\dot{p}\wedge \dot{q})+\ep}}
  = &
  \|h(\cdot-l)g(\frac{\cdot}{\lan k\ran^s}-k)\partial^{\g} \mu_k\|_{W^{\fy,\fy}_{n/(\dot{p}\wedge \dot{q})+\ep}}
  \\
  \lesssim &
  \|h(\cdot-l)\|_{W^{\fy,\fy}_{n/(\dot{p}\wedge \dot{q})+\ep}}
  \|g(\frac{\cdot}{\lan k\ran^s}-k)\partial^{\g} \mu_k\|_{W^{\fy,\fy}_{n/(\dot{p}\wedge \dot{q})+\ep}}\lesssim 1.
  \end{split}
  \ee
Combining with the above estimates yields that
\be
\begin{split}
  \|P_{-\d}e^{i\mu}\|_{W^{\fy,\fy}_{n/(\dot{p}\wedge \dot{q})+\ep}}\
  \sim &
  \sup_{k\in \bbZ^n}\langle k\rangle^{-\d}\|\s_ke^{i\mu}\|_{W^{\fy,\fy}_{n/(\dot{p}\wedge \dot{q})+\ep}}
  \\
  \lesssim &
  \sup_{k\in \bbZ^n}\langle k\rangle^{s\ep+sn/(\dot{p}\wedge \dot{q})-\d}\lesssim 1,
\end{split}
\ee
where we use the assumption \eqref{pf, 2,2} in the last inequality.
\end{proof}

\subsection{Interpolation with modulation case}
This subsection is devoted to the proof of Theorem \ref{thm, interpolation case}.
Thanks to the additional information of the second order derivative of $\mu$ derived by the prototype $|\xi|^{\b} (\b\in (1,2])$,
the conclusion in Theorem \ref{thm, wiener, potential loss} can be improved to be a more sharpened one. First, we establish
the following bounded result under the assumption associated with the second derivative of $\mu$.
This lemma is a slight generalization of \cite[Theorem 1.1]{NicolaTabacoo2018JPDOA}.
\begin{lemma}\label{lemma, modulation case}
  Suppose $0<p,q\leq \infty$.
  Let $\mu\in C^2(\bbR^n)$ be a real-valued function satisfying
  \be
  \partial^{\g}\mu \in W^{\fy,\fy}_{n/\dot{p}+\ep}\ (|\g|=2)
  \ee
  for some $\ep>0$.
  Then $e^{i\mu(D)}$ is bounded on $\mpq$.
\end{lemma}
\begin{proof}
  By the convolution relation
  \be
  \mpq \ast M^{\dot{p},\fy}\subset \mpq,
  \ee
  we only need to verify that $\scrF^{-1}e^{i\mu}\in M^{\dot{p},\fy}$,
  or equivalently, $e^{i\mu}\in W^{\fy,\dot{p}}$.
  Write
\be
\|e^{i\mu}\|_{W^{\fy,\dot{p}}}\sim \sup_{k\in \bbZ^n}\|\s_ke^{i\mu}\|_{W^{\fy,\dot{p}}}
=\sup_{k\in \bbZ^n}\|\s_0 e^{i\mu(\cdot+k)}\|_{W^{\fy,\dot{p}}}
=\sup_{k\in \bbZ^n}\|\s_0 e^{iR_k}\|_{W^{\fy,\dot{p}}},
\ee
where
\be
R_k(\xi):=\mu(k+\xi)-\mu(k)-\nabla \mu(k)\xi
=\sum_{|\g|=2}\frac{2\xi^{\g}}{\g!}\int_0^1(1-t)\partial^{\g}\mu(k+t\xi)dt.
\ee
Next,
\be
\|\s_0 e^{iR_k}\|_{W^{\fy,\dot{p}}}
\lesssim\|\s_0e^{iR_k}\|_{W^{\fy,\fy}_{n/\dot{p}+\ep}}=\|\s_0e^{i\s_0^*R_k}\|_{W^{\fy,\fy}_{n/\dot{p}+\ep}}.
\ee
Using Lemma \ref{lemma, conv of Wfy}, we further conclude that
\be
\begin{split}
\|\s_0e^{i\s_0^*R_k}\|_{W^{\fy,\fy}_{n/\dot{p}+\ep}}
  \lesssim &
  \|\s_0\|_{W^{\fy,\fy}_{n/\dot{p}+\ep}}\|e^{i\s_0^*R_k}\|_{W^{\fy,\fy}_{n/\dot{p}+\ep}}
  \\
  \lesssim &
  \|e^{i\s_0^*R_k}\|_{W^{\fy,\fy}_{n/\dot{p}+\ep}}
  \lesssim
  \exp\left(C\left\|\s_0^*R_k\right\|_{W^{\fy,\fy}_{n/\dot{p}+\ep}}\right).
\end{split}
\ee
We continue this estimate as follows:
\be
\begin{split}
  \|\s_0^*R_k\|_{W^{\fy,\fy}_{n/\dot{p}+\ep}}
  = &
  \left\|\s_0^*\sum_{|\g|=2}\frac{2\xi^{\g}}{\g!}
  \int_0^1(1-t)\partial^{\g}\mu(k+t\xi)dt\right\|_{W^{\fy,\fy}_{n/\dot{p}+\ep}}
  \\
  \lesssim &
  \sum_{|\g|=2}\left\|\s_0^*\xi^{\g}\right\|_{W^{\fy,\fy}_{n/\dot{p}+\ep}}\cdot
  \left\|\int_0^1(1-t)\partial^{\g}\mu(k+t\xi)dt\right\|_{W^{\fy,\fy}_{n/\dot{p}+\ep}}
  \\
  \lesssim &
  \sum_{|\g|=2}
  \left\|\int_0^1(1-t)\partial^{\g}\mu(k+t\xi)dt\right\|_{W^{\fy,\fy}_{n/\dot{p}+\ep}}
  \\
  \lesssim &
  \sum_{|\g|=2}
  \int_0^1 \left\|\partial^{\g} \mu(k+t\xi)\right\|_{W^{\fy,\fy}_{n/\dot{p}+\ep}}dt
  \\
  =&
  \sum_{|\g|=2}\int_0^1 \left\|\partial^{\g} \mu(t\xi)\right\|_{W^{\fy,\fy}_{n/\dot{p}+\ep}}dt
  \leq
  \sum_{|\g|=2}\left\|\partial^{\g} \mu\right\|_{W^{\fy,\fy}_{n/\dot{p}+\ep}}£¬
\end{split}
\ee
where in the last inequality we use the dilation property property in Lemma \ref{lemma, dila of Wfy}.

Combining the above estimates yields that
\be
\begin{split}
\|e^{i\mu}\|_{W^{\fy,\dot{p}}}\sim &\sup_{k\in \bbZ^n}\|\s_ke^{i\mu}\|_{W^{\fy,\dot{p}}}
\\
\lesssim &
\sup_{k\in \bbZ^n}\|\s_0e^{i\s_0^*R_k}\|_{W^{\fy,\fy}_{n/\dot{p}+\ep}}
\\
\lesssim &
\sup_{k\in \bbZ^n}\exp\left(C\left\|\s_0^*R_k\right\|_{W^{\fy,\fy}_{n/\dot{p}+\ep}}\right)
\lesssim
\exp\left(C\sum_{|\g|=2}\left\|\partial^{\g} \mu\right\|_{W^{\fy,\fy}_{n/\dot{p}+\ep}}\right)\lesssim 1.
\end{split}
\ee
\end{proof}

\begin{proof}[Proof of Theorem \ref{thm, interpolation case}]
Note that $M^{p,p}=W^{p,p}$. By Lemma \ref{lemma, modulation case}, we obtain
the boundedness of $e^{i\mu(D)}$ on $W^{p,p}$.
The desired conclusion then follows by the interpolation between this and the
boundedness of $e^{i\mu(D)}: W^{p,\fy}_{\d}\rightarrow \wpq (\d>sn/(\dot{p}\wedge \dot{q}))$ derived by Theorem \ref{thm, wiener, potential loss}.
\end{proof}

\subsection{An application to derivative condition}
We give the proof of Corollary \ref{coy, bd, exact conditions} in this subsection.
By establishing an embedding lemma, the conclusion in Corollary \ref{coy, bd, exact conditions} follows by
the boundedness result in Theorem \ref{thm, interpolation case}.
We state the embedding lemma as follows.

\begin{lemma}\label{lemma, derivative to wiener}
Let $N\in \mathbb{N}$,  $f\in \calC^N$, where $\calC^N$ is defined by
\be
\calC^N:=\{g\in C^N: |\partial^{\g}g|\lesssim 1\ \text{for all}\ |\g|\leq N\}.
\ee
with the norm $\|g\|_{\calC^N}:=\sum_{|\g|\leq N}\|\partial^{\g}g\|_{L^{\fy}}$.
We have $f\in W^{\fy,\fy}_N$ and
\be
\|f\|_{W^{\fy,\fy}_N}\lesssim \|f\|_{\calC^N}.
\ee
\end{lemma}
\begin{proof}
  Write
  \be
  \|f\|_{W^{\fy,\fy}_N}
  \sim
  \sup_{k\in \zn}\|\s_kf\|_{W^{\fy,\fy}_N}
  \sim
  \sup_{k\in \zn}\|\scrF^{-1}(\s_kf)(\cdot)\lan \cdot\ran^{N}\|_{L^{\fy}}
  \ee
  To estimate $\|\scrF^{-1}(\s_kf)(\cdot)\lan \cdot\ran^{N}\|_{L^{\fy}}$,
  we define the invariant derivative
  \be
  L_x:=\frac{x\cdot \nabla_{\xi}}{2\pi i|x|^2}.
  \ee
  Then
  \be
  L_x(e^{2\pi ix\cdot \xi})=e^{2\pi ix\cdot \xi}.
  \ee
  We also have
  \be
  L_x^*=-L_x.
  \ee
  Then
  \be
  \begin{split}
    \scrF^{-1}(\s_kf)(x)
    = &
    \int_{\rn}\s_k(\xi)f(\xi)(L_x)^N(e^{2\pi ix\cdot \xi})d\xi
    \\
    = &
    \int_{\rn}(L_x^*)^N(\s_k(\xi)f(\xi))e^{2\pi ix\cdot \xi}d\xi.
  \end{split}
  \ee
  From this, we further have
  \be
  \begin{split}
  |\scrF^{-1}(\s_kf)(x)|
  \lesssim &
  \int_{\rn}|(L_x^*)^N(\s_k(\xi)f(\xi))|d\xi
  \\
  \lesssim &
  |x|^{-N}\int_{\rn}\sum_{|\g_1|+|\g_2|=N}|\partial^{\g_1}\s_k(\xi)\partial^{\g_2}f(\xi)|d\xi
  \\
  \lesssim &
  |x|^{-N}\|f\|_{\calC^N}\int_{\rn}\sum_{|\g_1|+|\g_2|=N}|\partial^{\g_1}\s_k(\xi)|d\xi
  \lesssim \|f\|_{\calC^N}|x|^{-N}.
  \end{split}
  \ee
  This and the fact $|\scrF^{-1}(\s_kf)(x)|\lesssim \|\s_kf\|_{L^1}\lesssim \|f\|_{\calC^N}$ yield that
  \be
  |\scrF^{-1}(\s_kf)(x)|\lesssim \|f\|_{\calC^N}\lan x\ran^{-N}.
  \ee
  From this, we get the desired conclusion
  $$\|f\|_{W^{\fy,\fy}_N}
  \sim
  \sup_{k\in \zn}\|\scrF^{-1}(\s_kf)(\cdot)\lan \cdot\ran^{N}\|_{L^{\fy}}\lesssim \|f\|_{\calC^N}.$$
\end{proof}

Next, we recall a useful lemma for dealing with the multiplier near the origin.

\begin{lemma}[see \cite{Miyachi2009PAMS, Tomita2010}]\label{lemma, near the origin}
  Let $0<p\leq \infty$, $\ep>n(1/\dot{p}-1)$.
  Let $\mu$ be a $C^{[n(1/\dot{p}-1/2)]+1}(\mathbb{R}^n \backslash \{0\})$ function with compact support, satisfying
\begin{equation}
|\partial^{\gamma}\mu(\xi)|\leq C_{\gamma}|\xi|^{\epsilon-|\gamma|}, \hspace{5mm} |\xi|\neq 0 ,\ \ |\gamma|\leq[n(1/\dot{p}-1/2)]+1,
\end{equation}
then $m\in \scrF L^{\dot{p}}$.
\end{lemma}

\begin{proof}[Proof of Corollary \ref{coy, bd, exact conditions}]
  Let $\r_0$ be a smooth function supported on B(0,1) and
  satisfies $\r_0(\xi)=1$ on $B(0,1/2)$. Denote by
  \be
  \mu_1:= \r_0\mu,\ \ \ \mu_2:= (1-\r_0)\mu.
  \ee
  Take $\r_0^*$ to be a smooth function supported on B(0,2), satisfying that $\r_0\cdot \r_0^*=\r_0$.
  Note that
  \be
  |\partial^{\g}[\r_0^*(e^{i\mu_1}-1](\xi)|\lesssim |\xi|^{\ep-|\g|},\ |\xi|\neq 0,\ \ |\gamma|\leq[n/(1/\dot{p}-1/2)]+1.
  \ee
  Using Lemma \ref{lemma, near the origin}, we have $\r_0^*(e^{i\mu_1}-1) \in \scrF L^{\dot{p}}$.
  Then,
  \be
  e^{i\mu_1}=\r_0^*(e^{i\mu_1}-1)+\r_0^*\in \scrF L^{\dot{p}}.
  \ee
  Moreover, observing that $e^{i\mu_1}$ have compact support, we further have
  \be
 \|\scrF^{-1}e^{i\mu_1}\|_{W^{\dot{p},\fy}}\sim
 \|e^{i\mu_1}\|_{M^{\fy,\dot{p}}}\sim \|e^{i\mu_1}\|_{\scrF L^{\dot{p}}}\lesssim 1.
  \ee
  It follows by $\wpq\ast W^{\dot{p},\fy}\subset \wpq$
  that
  $e^{i\mu_1(D)}$ is bounded on $\wpq$.

  Now, we turn to deal with $e^{i\mu_2}$.\\
  If $\b\in (0,1]$, we have,
  \be
  |\partial^{\g}\mu_2(\xi)|\lesssim 1,\ \ \ \ 1\leq|\gamma|\leq[n/(\dot{p}\wedge \dot{q})]+2.
  \ee
  Thus
  \be
  \nabla\mu_2\in (\calC^{[n/(\dot{p}\wedge \dot{q})]+1})^n
  \ee
  There exists a small positive constant $\ep$ such that $n/(\dot{p}\wedge \dot{q})+\ep\leq [n/(\dot{p}\wedge \dot{q})]+1$ and
   \be
   W^{\fy,\fy}_{[n/(\dot{p}\wedge \dot{q})]+1}\subset W^{\fy,\fy}_{n/(\dot{p}\wedge \dot{q})+\ep}.
   \ee
  From this and the embedding relation
  $\calC^{[n/(\dot{p}\wedge \dot{q})]+1}\subset W^{\fy,\fy}_{[n/(\dot{p}\wedge \dot{q})]+1}$
  from Lemma \ref{lemma, derivative to wiener}, we further deduce that
  \be
  \nabla\mu_2\in (\calC^{[n/(\dot{p}\wedge \dot{q})]+1})^n\subset (W^{\fy,\fy}_{[n/(\dot{p}\wedge \dot{q})]+1})^n
  \subset  (W^{\fy,\fy}_{n/(\dot{p}\wedge \dot{q})+\ep})^n.
  \ee
  It then follows by Theorem \ref{thm, wiener, potential persistence} that $e^{i\mu_2(D)}$ is bounded on $\wpq$.\\
  If $\b\in (1,2]$, we have
  \be
  |\partial^{\g}\mu_2(\xi)|\lesssim |\xi|^{\b-|\g|},\ \ \ \ 1\leq|\gamma|\leq[n/\dot{p}]+3.
  \ee
  Denote by $s:=\b-1$. We have
  \be
  \lan \xi\ran^{-s}\nabla\mu_2\in (\calC^{[n/(\dot{p}\wedge \dot{q})]+1})^n,
  \ \
  \partial^{\g}\mu_2 \in \calC^{[n/(\dot{p}\wedge \dot{q})]+1}\ (|\g|=2).
  \ee
  Hence, there exists a small positive constant $\ep$ such that
  \be
  \lan \xi\ran^{-s}\nabla\mu_2\in (W^{\fy,\fy}_{n/(\dot{p}\wedge \dot{q})+\ep})^n,
  \ \
  \partial^{\g}\mu_2 \in W^{\fy,\fy}_{n/(\dot{p}\wedge \dot{q})+\ep}\subset W^{\fy,\fy}_{n/\dot{p}+\ep}\  (|\g|=2).
  \ee
  This and Theorem \ref{thm, interpolation case} imply the boundedness of $e^{i\mu_2(D)}$
  from $W^{p,q}_{\d}$ to $\wpq$, where $\d\geq sn|1/p-1/q|=n(\b-1)|1/p-1/q|$
  with strict inequality when $p\neq q$.

  Combining the above two cases we have that
  $e^{i\mu_2(D)}: W^{p,q}_{\d} \rightarrow \wpq$ is bounded, where $\d\geq n|1/p-1/q|\max\{\b-1,0\}$
  with strict inequality when $\b>1$, $p\neq q$.
  This and the boundedness of $e^{i\mu_1(D)}$ yield the final conclusion of Corollary \ref{coy, bd, exact conditions}.
\end{proof}

\section{sharp loss of potential}
This section is devoted to the proof of Theorem \ref{thm, wiener, sharp potential loss}.
To this end, we first develop the scattered property of $\nabla\mu$ in Lemma \ref{lemma, scattered},
and then use this to establish some useful estimates of certain special functions in Lemma \ref{lemma, estimates of special}.
Thanks to these estimates and a rotation trick,
the boundedness results on Wiener amalgam spaces can be reduced to the simple embedding relations of weighted sequences.
This yields the final conclusion in Theorem \ref{thm, wiener, sharp potential loss}.

\begin{lemma}[Scattered set]\label{lemma, scattered}
 Let $1<\b\leq 2$, and let $\mu$ be a real-valued $C^2(\bbR^n\bs \{0\})$ function.
 Suppose that the Hessian determinant of $\mu$ is not zero at some point $\k_0$ with $|\k_0|=1$.
  Moreover,
  \be
  \mu(\la\xi)=\la^{\b}\mu(\xi),\ \ \ \la\geq 1,\ \ \xi\in B(\k_0, r_0)\cap \bbS^{n-1}
  \ee
  for some $r_0>0$.
  Then there exists a cone $\G$ with vertex at the origin such that
  the set $\{\nabla \mu(\xi_l)\}_{l\in E}$ is scattered in the following sense:
  \be
  |\nabla \mu(\xi_k)-\nabla \mu(\xi_l)|\geq C>0
  \ee
for $k,l\in E$, $k\neq l$, where
\be
E:= \{l\in \bbZ^n\bs\{0\}: \langle l\rangle^{\frac{2-\b}{\b-1}}l\in \G\},\ \xi_l:=\langle l\rangle^{\frac{2-\b}{\b-1}}l.
\ee
\end{lemma}
\begin{proof}
Note that the Hessian matrix of $\mu$ is a real symmetric matrix.
This and the assumption
$|\hes\mu(\k_0)|\neq 0$ imply that all the eigenvalues
of $\hes\mu(\k_0)$, denoted by $\la_j,j=1,2,\cdots,n$, are nonzero real numbers.
Write
\be
\hes\mu(\k_0)=P^{-1}diag(\la_1,\la_2,\cdots,\la_n)P,
\ee
where $P$ is a orthogonal matrix.
Take
\be
B=diag(sgn\la_1, sgn\la_2,\cdots, sgn\la_n),\ \ Q=P^{-1}BP.
\ee
Denote by $A(\xi):=Q(\hes\mu(\xi))$. Then
\be
\begin{split}
  A(\k_0):=Q(\hes\mu(\k_0))= &P^{-1}B(diag(\la_1,\la_2,\cdots,\la_n))P
  \\
  = &P^{-1}(diag(|\la_1|,|\la_2|,\cdots,|\la_n|))P.
\end{split}
\ee
By the continuity of $\partial_{i,j}\mu$, we have
\ben\label{pf, 3,1}
\lim_{\xi\rightarrow 0}\|\hes\mu(\xi)-\hes\mu(\k_0)\|=0,
\een
where the norm of a matrix $A$ is defined by
\begin{equation*}
  \|A\|:=\left(\sum_{i=1}^m\sum_{j=1}^n|a_{ij}|^2\right)^{1/2}.
\end{equation*}

For a nonzero vector $x\in \rn$,
\be
\begin{split}
  x^TA(\k_0)x
  =&
  x^TP^{-1}(diag(|\la_1|,|\la_2|,\cdots,|\la_n|))Px
  \\
  = &
  (Px)^T(diag(|\la_1|,|\la_2|,\cdots,|\la_n|))Px
  \\
  = &
  \sum_{j=1}^n|\la_j|\cdot|(Px)_j|^2\geq \min_{1\leq j\leq n}|\la_j|\cdot|Px|^2=\min_{1\leq j\leq n}|\la_j|\cdot|x|^2
\end{split}
\ee
and
\be
\begin{split}
   |x^T(A(\xi)-A(\k_0))x|
 = &
 |x^TQ(\text{Hess}\mu(\xi)-\text{Hess}\mu(\k_0))x|
 \\
 \lesssim & \|Q(\text{Hess}\mu(\xi)-\text{Hess}\mu(\k_0))\|\cdot|x|^2
 \\
 \lesssim & \|\text{Hess}\mu(\xi)-\text{Hess}\mu(\k_0)\|\cdot|x|^2.
\end{split}
\ee
From the above two estimates and \eqref{pf, 3,1}, we can choose a small constant $r\in (0, r_0]$ such that
for $\xi\in \bbS^{n-1}\cap B(\k_0,r)$,
\ben\label{pf, 3,2}
\begin{split}
  x^TA(\xi)x\geq &x^TA(\k_0)x-|x^T(A(\xi)-A(\k_0))x|
  \\
  \geq &
  \min_{1\leq j\leq n}|\la_j|\cdot|x|^2-|x^T(A(\xi)-A(\k_0))x|
  \\
  \geq &
  \frac{1}{2}\min_{1\leq j\leq n}|\la_j|\cdot|x|^2.
\end{split}
\een
Denote by $\G$ the cone containing all vectors $\xi$ such that $\frac{\xi}{|\xi|}\in \bbS^{n-1}\cap B(\k_0,r)$.
By the assumption, we know that for every $|\g|=2$,
\be
\partial^{\g}\mu(\la\xi)= \la^{\b-2}\partial^{\g}\mu(\xi),\ \ \la\geq 1,\ \  \xi\in \bbS^{n-1}\cap B(\k_0,r),
\ee
which implies that
\be
A(\la\xi)=Q(\text{Hess}\mu(\la\xi))= \la^{\b-2}Q\text{Hess}\mu(\xi)=\la^{\b-2}A(\xi)
,\ \ \la\geq 1,\ \  \xi\in \bbS^{n-1}\cap B(\k_0,r).
\ee
From this and \eqref{pf, 3,2}, we know that for $\xi\in \G\bs B(0,1)$
\ben\label{pf, 3,3}
x^TA(\xi)x\geq |\xi|^{\b-2}\frac{1}{2}\min_{1\leq j\leq n}|\la_j|\cdot|x|^2\gtrsim |\xi|^{\b-2}|x|^2.
\een
Set
\be
E:= \{l\in \bbZ^n\bs\{0\}: \langle l\rangle^{\frac{2-\b}{\b-1}}l\in \G\},\ \  \xi_l:=\langle l\rangle^{\frac{2-\b}{\b-1}}l.
\ee
Next, we will proof that $\{\nabla \mu(\xi_l)\}_{l\in A}$ is a scattered set in the sense that
\be
|\nabla \mu(\xi_k)-\nabla \mu(\xi_l)|\geq C,\ \ \ k,l\in A,\ k\neq l.
\ee
For any two fixed point $\xi_l,\xi_k$
we set
\be
F(\th)=\nabla\mu(\xi_l+\th(\xi_k-\xi_l))\cdot (Q(\xi_k-\xi_l)).
\ee
A direct calculation yields that
\be
\begin{split}
  F'(\th)
  = &
  \big(\sum_{j=1}^n\partial_j\mu(\xi_l+\th(\xi_k-\xi_l))(Q(\xi_k-\xi_l))_j\big)'
  \\
  = &
  \sum_{i=1}^n\big(\sum_{j=1}^n\partial_{i,j}\mu(\xi_l+\th(\xi_k-\xi_l))(Q(\xi_k-\xi_l))_j\big)\cdot (\xi_k-\xi_l)_i
  \\
  = &
  \sum_{i=1}^n\sum_{j=1}^n\partial_{i,j}\mu(\xi_l+\th(\xi_k-\xi_l))(Q(\xi_k-\xi_l))_j\cdot (\xi_k-\xi_l)_i
  \\
  = &
  (Q(\xi_k-\xi_l))^T\text{Hess}\mu(\xi_k+\th(\xi_k-\xi_l))(\xi_k-\xi_l).
\end{split}
\ee
Observe that $(Q(\xi_k-\xi_l))^T=(\xi_k-\xi_l)^TQ^T=(\xi_k-\xi_l)^TQ$, we have
\be
\begin{split}
  F'(\th)= &(\xi_k-\xi_l)^TQ\text{Hess}\mu(\xi_l+\th(\xi_k-\xi_l))(\xi_k-\xi_l)
  \\
  = &
  (\xi_k-\xi_l)^TA(\xi_l+\th(\xi_k-\xi_l))(\xi_k-\xi_l).
\end{split}
\ee
Note that $\xi_l+\th(\xi_k-\xi_l)\in \G\bs B(0,1)$. We use \eqref{pf, 3,3} to further obtain
\be
F'(\th)\gtrsim |\xi_l+\th(\xi_k-\xi_l)|^{\b-2}|\xi_k-\xi_l|^2.
\ee
Next,
\ben\label{pf, 3,4}
\begin{split}
  F(1)-F(0)= &\int_0^1 F'(\th)d\th
  \\
  \gtrsim &
  |\xi_k-\xi_l|^2\int_0^1|\xi_l+\th(\xi_k-\xi_l)|^{\b-2}d\th.
\end{split}
\een
By the following identity
\be
4\left(\frac{|\xi_l|}{2|\xi_k-\xi_l|}\right)+\left(1-\frac{2|\xi_l|}{|\xi_k-\xi_l|}\right)=1,
\ee
we have
\be
\max\left\{4\left(\frac{|\xi_l|}{2|\xi_k-\xi_l|}\right), 1-\frac{2|\xi_l|}{|\xi_k-\xi_l|}\right\}\geq 1/2.
\ee
If $4\left(\frac{|\xi_l|}{2|\xi_k-\xi_l|}\right)\geq 1/2$, we have
\be
|\xi_l|\geq 2\th|\xi_k-\xi_l|,\ \ (\th\leq 1/8).
\ee
Then
\ben\label{pf, 3,5}
\begin{split}
  \int_0^1|\xi_l+\th(\xi_k-\xi_l)|^{\b-2}d\th
  \geq &
  \int_0^{1/8}|\xi_l+\th(\xi_k-\xi_l)|^{\b-2}d\th
  \\
  \gtrsim &
  \int_0^{1/8}|\xi_l|^{\b-2}d\th\sim |\xi_l|^{\b-2}.
\end{split}
\een
If $1-\frac{2|\xi_l|}{|\xi_k-\xi_l|}\geq 1/2$, we have
\be
\th|\xi_k-\xi_l|\geq 2|\xi_l|,\ \ \ (1/2\leq \th\leq 1).
\ee
Then
\ben\label{pf, 3,6}
\begin{split}
  \int_0^1|\xi_l+\th(\xi_k-\xi_l)|^{\b-2}d\th
  \geq &
  \int_{1/2}^{1}|\xi_l+\th(\xi_k-\xi_l)|^{\b-2}d\th
  \\
  \gtrsim &
  \int_{1/2}^{1}|\th(\xi_k-\xi_l)|^{\b-2}d\th
  \gtrsim
  |\xi_k-\xi_l|^{\b-2}.
  \end{split}
\een
The combination of \eqref{pf, 3,4}, \eqref{pf, 3,5} and \eqref{pf, 3,6} yields that
\be
  F(1)-F(0)\gtrsim |\xi_k-\xi_l|^2\min\{|\xi_l|^{\b-2}, |\xi_k-\xi_l|^{\b-2}\}.
\ee
On the other hand
\be
\begin{split}
  |F(1)-F(0)|
  = &
  |(\nabla \mu(\xi_k)-\nabla \mu(\xi_l))\cdot (Q(\xi_k-\xi_l))|
  \\
  \leq &
  |(\nabla \mu(\xi_k)-\nabla \mu(\xi_l))|\cdot |Q(\xi_k-\xi_l)|
  \\
  = &
  |(\nabla \mu(\xi_k)-\nabla \mu(\xi_l))|\cdot |\xi_k-\xi_l|.
\end{split}
\ee
It follows by the above two estimates that
\be
  |(\nabla \mu(\xi_k)-\nabla \mu(\xi_l))|\gtrsim |\xi_k-\xi_l|\min\{|\xi_l|^{\b-2}, |\xi_k-\xi_l|^{\b-2}\}.
\ee
Note $\b\in (1,2]$.
Take $\al=2-\b\in [0,1)$ in Lemma \ref{lemma, position}, there exists a constant $\d>0$ such that
\be
B(\xi_l, \d|\xi_l|^{2-\b})\cap B(\xi_k, \d|\xi_k|^{2-\b})=\emptyset,
\ee
which implies that
\be
|\xi_k-\xi_l|\gtrsim |\xi_l|^{2-\b}.
\ee
Hence,
\be
\begin{split}
  |\xi_k-\xi_l|\min\{|\xi_l|^{\b-2}, |\xi_k-\xi_l|^{\b-2}\}
  \gtrsim &
  \min\{|\xi_k-\xi_l||\xi_l|^{\b-2}, |\xi_k-\xi_l|^{\b-1}\}
  \\
  \gtrsim &
  \min\{|\xi_l|^{2-\b}|\xi_l|^{\b-2}, |\xi_k-\xi_l|^{\b-1}\}
  \\
  = &
  \min\{1, |\xi_k-\xi_l|^{\b-1}\}\gtrsim 1.
\end{split}
\ee
This completes the proof of Lemma \ref{lemma, scattered}.
\end{proof}

\begin{lemma}[Estimates of special functions]\label{lemma, estimates of special}
  Let $0<p, q\leq \infty$, $s\in \bbR$.
  Let $\mu$ be a real-valued functions satisfying the assumptions of Theorem \ref{thm, wiener, sharp potential loss},
  $E$ be the set mentioned in Lemma \ref{lemma, scattered}. For every
  fixed nonnegative truncated (only finite nonzero items) sequence $\{a_k\}_{k\in E}$, there exists
  a Schwartz function $F$ corresponding to $\{a_k\}_{k\in E}$, such that
  the following two estimates
  \bn
  \item $\|e^{i\mu(D)}F\|_{W^{p,q}}\sim \|\{a_k\}_{k\in E}\|_{l^p}$,
  \item $\|F\|_{W^{p,q}_s}\sim \|\{a_k\}_{k\in E}\|_{l^q_{s/(\b-1)}}$,
  \en
  hold uniformly for all $\{a_k\}_{k\in E}$.
\end{lemma}
\begin{proof}
  We only give the proof for $p,q<\infty$, since the other cases can be handled similarly with some slight modification.
  Denote by $\xi_k:=\langle k\rangle^{\frac{2-\b}{\b-1}}k$.
  It follows by Lemma \ref{lemma, scattered} that
  there exists a constant $R>0$ such that the family $\{B(\nabla\mu(\xi_k), R)\}_{k\in E}$ is pairwise disjoint.
  In additional, by Lemma \ref{lemma, position}, there exists a positive constant $r<1/2$ such that
  the family $\{B(\xi_k, r)\}_{k\in \zn}$ is pairwise disjoint.
Take $\widehat{h}, \widehat{h^*}$ be two nonnegative real-valued radial smooth functions satisfying
\be
\widehat{h}(\xi)=\widehat{h^{*}}(\xi)=1\ \text{on}\ B(0,r/8),\ \
\text{supp}\widehat{h}\subset \text{supp}\widehat{h^*}\subset B(0,r/4),\ \ \widehat{h}\cdot \widehat{h^*}=\widehat{h}.
\ee
Denote by $\widehat{h_k}(\xi):= \widehat{h}(\xi-\xi_k)$ and $\widehat{h^*_k}(\xi):= \widehat{h^*}(\xi-\xi_k)$.
For any nonnegative truncated (only finite nonzero items) sequence $\{a_k\}_{k\in A}$, we set
\be
\widehat{F}:=\sum_{k\in A}a_k\widehat{h_k}.
\ee
Write
\be
\begin{split}
  e^{i\mu}\widehat{F}
  = &
  e^{i\mu}\sum_{k\in A}a_k\widehat{h_k}
  =
  e^{i\mu}\sum_{k\in A}a_k\widehat{h_k}\widehat{h^*_k}
  \\
  = &
  \left(\sum_{k\in A}\widehat{h^*_k}(\xi)e^{i\mu(\xi)-i\nabla \mu(\xi_k)\xi}\right)
  \cdot \left(\sum_{k\in A}a_k\widehat{h_k}(\xi)e^{i\nabla\mu(\xi_k)\xi}\right)
  \\
  = :& m(\xi)\cdot \widehat{G}(\xi).
\end{split}
\ee
We claim that $\|m\|_{W^{\fy,\fy}_{n/(\dot{p}\wedge \dot{q})+\ep}}\lesssim 1$.
As in the proof of Theorem \ref{thm, wiener, potential persistence}, we write
\be
\begin{split}
  \|m\|_{W^{\fy,\fy}_{n/(\dot{p}\wedge \dot{q})+\ep}}
  = &
  \sup_{k\in A}\|\widehat{h^*_k}(\xi)e^{i\mu(\xi)-i\nabla \mu(\xi_k)\xi}\|_{W^{\fy,\fy}_{n/(\dot{p}\wedge \dot{q})+\ep}}
  \\
  = &
  \sup_{k\in A}\|\widehat{h^*_k}(\xi)e^{i\mu(\xi)-i\nabla \mu(\xi_k)\xi-i\mu(\xi_k)}\|_{W^{\fy,\fy}_{n/(\dot{p}\wedge \dot{q})+\ep}}
  \\
  = &
  \sup_{k\in A}\|\widehat{h^*}(\xi)e^{i\mu(\xi+\xi_k)-i\nabla \mu(\xi_k)\xi-i\mu(\xi_k)}\|_{W^{\fy,\fy}_{n/(\dot{p}\wedge \dot{q})+\ep}}
  \\
  = &
  \sup_{k\in A}\|\widehat{h^*}(\xi)e^{iR_k(\xi)}\|_{W^{\fy,\fy}_{n/(\dot{p}\wedge \dot{q})+\ep}}.
\end{split}
\ee
where
\be
R_k(\xi)=\mu(\xi_k+\xi)-\mu(\xi_k)-\nabla \mu(\xi_k)\xi
=\sum_{|\g|=2}\frac{2\xi^{\g}}{\g!}\int_0^1(1-t)\partial^{\g}\mu(\xi_k+t\xi)dt.
\ee
Take $\psi$ to be a $C_c^{\fy}(\rn)$ function supported on $B(0,1/2)$
such that $\psi=1$ on $B(0,1/4)$,
then $\psi\widehat{h^*}=\widehat{h^*}$.
By the similar argument as in the proof of Theorem \ref{thm, wiener, potential persistence},
\be
\begin{split}
  &\|\widehat{h^*}(\xi)e^{iR_k(\xi)}\|_{W^{\fy,\fy}_{n/(\dot{p}\wedge \dot{q})+\ep}}
  =
  \|\widehat{h^*}(\xi)e^{i\psi(\xi)R_k(\xi)}\|_{W^{\fy,\fy}_{n/(\dot{p}\wedge \dot{q})+\ep}}
  \\
  = &
  \left\|\widehat{h^*}(\xi)
  \exp\left(i\psi(\xi)\sum_{|\g|=2}\frac{2\xi^{\g}}{\g!}\int_0^1(1-t)\psi(t\xi)\partial^{\g}\mu(\xi_k+t\xi)dt\right)\right\|_{W^{\fy,\fy}_{n/(\dot{p}\wedge \dot{q})+\ep}}
  \\
  \lesssim &
  \exp\left(C\left\|\psi(\xi)\sum_{|\g|=2}\frac{2\xi^{\g}}{\g!}\int_0^1(1-t)\psi(t\xi)\partial^{\g}\mu(\xi_k+t\xi)dt\right\|_{W^{\fy,\fy}_{n/(\dot{p}\wedge \dot{q})+\ep}}\right)
  \\
  \lesssim &
  \exp\left(C\left\|\sum_{|\g|=2}\int_0^1(1-t)\psi(t\xi)\partial^{\g}\mu(\xi_k+t\xi)dt\right\|_{W^{\fy,\fy}_{n/(\dot{p}\wedge \dot{q})+\ep}}\right)
  \\
  \lesssim &
  \exp\left(C\sum_{|\g|=2}\|\psi(\cdot-\xi_k)\partial^{\g}\mu(\cdot)\|_{W^{\fy,\fy}_{n/(\dot{p}\wedge \dot{q})+\ep}}\right)
\end{split}
\ee
Together with this and the following estimate
\be
 \begin{split}
   \sum_{|\g|=2}\|\psi(\cdot-\xi_k)\partial^{\g}\mu(\cdot)\|_{W^{\fy,\fy}_{n/(\dot{p}\wedge \dot{q})+\ep}}
   = &
   \sum_{|\g|=2}\|\psi(\cdot-\xi_k)(1-\rho_0)\partial^{\g}\mu(\cdot)\|_{W^{\fy,\fy}_{n/(\dot{p}\wedge \dot{q})+\ep}}
   \\
   \lesssim &
   \sum_{|\g|=2}\|\psi(\cdot-\xi_k)\|_{W^{\fy,\fy}_{n/(\dot{p}\wedge \dot{q})+\ep}}
   \cdot \|(1-\rho_0)\partial^{\g}\mu(\cdot)\|_{W^{\fy,\fy}_{n/(\dot{p}\wedge \dot{q})+\ep}}
   \\
   = &
   \sum_{|\g|=2}\|\psi\|_{W^{\fy,\fy}_{n/(\dot{p}\wedge \dot{q})+\ep}}
   \cdot \|(1-\rho_0)\partial^{\g}\mu\|_{W^{\fy,\fy}_{n/(\dot{p}\wedge \dot{q})+\ep}}
   \\
   \lesssim & \sum_{|\g|=2}\|(1-\rho_0)\partial^{\g}\mu\|_{W^{\fy,\fy}_{n/(\dot{p}\wedge \dot{q})+\ep}},
 \end{split}
\ee
the desired estimate follows by
\be
 \begin{split}
   \|m\|_{W^{\fy,\fy}_{n/(\dot{p}\wedge \dot{q})+\ep}}
   = &
   \sup_{k\in \zn}\|\widehat{h^*}(\xi)e^{iR_k(\xi)}\|_{W^{\fy,\fy}_{n/(\dot{p}\wedge \dot{q})+\ep}}
   \\
   \lesssim &
   \sup_{k\in \zn}\exp\left(C\sum_{|\g|=2}\|\psi(\cdot-\xi_k)\partial^{\g}\mu(\cdot)\|_{W^{\fy,\fy}_{n/(\dot{p}\wedge \dot{q})+\ep}}\right)
   \\
   \lesssim &
   \sup_{k\in \zn}\exp\left(C\sum_{|\g|=2}\|(1-\rho_0)\partial^{\g}\mu\|_{W^{\fy,\fy}_{n/(\dot{p}\wedge \dot{q})+\ep}}\right)\lesssim 1,
 \end{split}
\ee
where in the last estimate we use the assumption that $(1-\r_0)\partial^{\g}\mu \in W^{\fy,\fy}_{n/(\dot{p}\wedge \dot{q})+\ep}$ for $|\g|=2$.
Hence,
\ben\label{pf, 3,7}
\begin{split}
  \|e^{i\mu(D)}F\|_{\wpq}
  = &
  \|m(D)G\|_{\wpq}
  \\
  \lesssim &
  \|\scrF^{-1}m\|_{W^{\dot{p}\wedge \dot{q},\fy}}\|G\|_{\wpq}
  \\
  = &
  \|m\|_{M^{\fy,\dot{p}\wedge \dot{q}}}\|G\|_{\wpq}\lesssim \|m\|_{W^{\fy,\fy}_{n/(\dot{p}\wedge \dot{q})+\ep}}\|G\|_{\wpq}\lesssim \|G\|_{\wpq}.
\end{split}
\een
On the other hand, observe
\be
\begin{split}
  \widehat{G}
  = &
  \sum_{k\in A}a_k\widehat{h_k}(\xi)e^{i\nabla\mu(\xi_k)\xi}
  \\
  = &
  \left(\sum_{k\in A}\widehat{h^*_k}(\xi)e^{-i\mu(\xi)+i\nabla \mu(\xi_k)\xi}\right)\cdot
  \left(e^{i\mu}\sum_{k\in A}a_k\widehat{h_k}\right)=\overline{m}e^{i\mu}\widehat{F}.
\end{split}
\ee
A similar argument yields that
\be
\|\overline{m}\|_{W^{\fy,\fy}_{n/(\dot{p}\wedge \dot{q})+\ep}}\lesssim 1,
\ee
and
\ben\label{pf, 3,8}
\|G\|_{\wpq}\lesssim \|\overline{m}\|_{W^{\fy,\fy}_{n/(\dot{p}\wedge \dot{q})+\ep}}\|e^{i\mu(D)}F\|_{\wpq}
\lesssim \|e^{i\mu(D)}F\|_{\wpq}.
\een
Then  $\|G\|_{\wpq}\sim \|e^{i\mu(D)}F\|_{\wpq}$ follows by \eqref{pf, 3,7} and \eqref{pf, 3,8}.
Let us turn to the estimate of $\|G\|_{\wpq}$.

\textbf{Lower estimate of $\|G\|_{\wpq}$.}
Take $\widehat{\phi}$ to be a real-valued radial function supported on $B(0,3r/4)$ and satisfying
\be
\widehat{\phi}(\xi)=1,\ \ \xi\in B(0,r/2).
\ee
Recalling $\text{supp}\widehat{h_k}\subset B(\xi_k,r/4)$, and the disjoint property of $\{B(\xi_k, r)\}_{k\in \zn}$,
for $\xi\in B(\xi_l, r/4)$ we have
\be
\widehat{\phi}(\eta-\xi)\widehat{h_k}(\eta)\ \text{equals to}\ \widehat{h_l}(\eta)\ \ \text{when}\  k=l,\ \text{and vanishes}\ \text{when}\  k\neq l.
\ee
Hence, for $\xi\in B(\xi_l, r/4)$,
\be
\begin{split}
  \widehat{G}(\eta)\widehat{\phi}(\eta-\xi)
  = &
  \sum_{k\in A}a_k\widehat{h_k}(\eta)e^{i\nabla\mu(\xi_k)\eta}\widehat{\phi}(\eta-\xi)
  \\
  = &
  a_l\widehat{h_l}(\eta)e^{i\nabla\mu(\xi_l)\eta}
\end{split}
\ee
From this and a direct calculation, for $\xi\in B(\xi_l, r/4)$ we further have
\be
\begin{split}
  \left|V_{\phi}G(x,\xi)\right|
  = &
  \left|\int_{\rn}\widehat{G}(\eta)\widehat{\phi}(\eta-\xi)e^{2\pi i\eta\cdot x}d\eta\right|
  \\
  = &
  \left|a_l\int_{\rn}\widehat{h_l}(\eta)e^{i\nabla\mu(\xi_l)\eta}e^{2\pi i\eta\cdot x}d\eta\right|
  \\
  = &
  \left|a_l\int_{\rn}\widehat{h}(\eta)e^{i\nabla\mu(\xi_l)\eta}e^{2\pi i\eta\cdot x}d\eta\right|
  =
  a_l|h(x+\nabla\mu(\xi_l)/2\pi)|.
\end{split}
\ee
It follows that
\be
\begin{split}
  \left\|V_{\phi}G(x,\xi)\right\|_{L^q_{\xi}}
  \geq &
  \left(\sum_{k\in A}\int_{B(\xi_k,r/4)}\left|V_{\phi}G(x,\xi)\right|^qd\xi\right)^{1/q}
  \\
  \sim &
  \left(\sum_{k\in A}a_k^q|h(x+\nabla\mu(\xi_k)/2\pi)|^q\right)^{1/q}.
\end{split}
\ee
By Lemma \ref{lemma, scattered}, we know that
the family $\{B(-\nabla\mu(\xi_l)/2\pi, R/2\pi)\}_{k\in A}$ is pairwise disjoint.
Then the desired lower estimate follows by
\be
\begin{split}
  \|G\|_{\wpq}
  = &
  \|\left\|V_{\phi}G(x,\xi)\right\|_{L^q_{\xi}}\|_{L^p_x}
  \\
  \gtrsim &
  \left\|\left(\sum_{k\in A}a_k^q|h(x+\nabla\mu(\xi_k)/2\pi)|^q\right)^{1/q}\right\|_{L^p_x}
  \\
  \geq &
  \left(\sum_{l\in A}\int_{B(-\nabla\mu(\xi_l)/2\pi, R/2\pi)}\left(\sum_{k\in \zn}a_k^q|h(x+\nabla\mu(\xi_k)/2\pi)|^q\right)^{p/q}\right)^{1/p}
  \\
  \geq &
  \left(\sum_{l\in A}\int_{B(-\nabla\mu(\xi_l)/2\pi, R/2\pi)}\left(a_l^q|h(x+\nabla\mu(\xi_l)/2\pi)|^q\right)^{p/q}\right)^{1/p}
  \\
  = &
  \left(\sum_{l\in A}\int_{B(0, R/2\pi)}a_l^p|h(x)|^p\right)^{1/p}\sim \|\{a_k\}_{k\in A}\|_{l^p}.
\end{split}
\ee

\textbf{Upper estimate of $\|G\|_{\wpq}$.}
By the definition of $h_k$ and $\phi$, we know that for any $k\in A$,
\be
\{\xi\in \rn: \widehat{h_k}\widehat{\phi}(\eta-\xi)\neq 0\}\subset B(\xi_k, r).
\ee
Recalling that the family $\{B(\xi_k,r)\}_{k\in \zn}$ is pairwise disjoint,
for $\xi\in B(\xi_k,r)$ we have,
\be
\begin{split}
  \left|V_{\phi}G(x,\xi)\right|
  = &
  \left|\int_{\rn}\widehat{G}(\eta)\widehat{\phi}(\eta-\xi)e^{2\pi i\eta\cdot x}d\eta\right|
  \\
  = &
  \left|a_k\int_{\rn}\widehat{h_k}(\eta)\widehat{\phi}(\eta-\xi)e^{i\nabla\mu(\xi_k)\eta}e^{2\pi i\eta\cdot x}d\eta\right|
  \\
  = &
  \left|a_k\int_{\rn}\widehat{h_k}(\eta)\widehat{M_{\xi}\phi}(\eta)e^{i\nabla\mu(\xi_k)\eta}e^{2\pi i\eta\cdot x}d\eta\right|
  \\
  = &
  a_k\left|h_k\ast M_{\xi}\phi(x+\nabla\mu(\xi_k)/2\pi)\right|
  \\
  \lesssim &
  a_k\left|(h_k|\ast |M_{\xi}\phi|)(x+\nabla\mu(\xi_k)/2\pi)\right|
  \\
  = &
  a_k\left|h|\ast |\phi|(x+\nabla\mu(\xi_k)/2\pi)\right|\lesssim a_k\lan x+\nabla\mu(\xi_k)/2\pi\ran^{-\scrL},
\end{split}
\ee
where in the last inequality we use the fact that both $h$ and $\phi$ are Schwartz functions, $\scrL$ indicates a sufficiently large number.
It follows that
\ben\label{pf, 3,10}
\begin{split}
  \left\|V_{\phi}G(x,\xi)\right\|_{L^q}
  = &
  \left(\sum_{k\in E}\int_{B(\xi_k,r)}\left|V_{\phi}G(x,\xi)\right|^qd\xi\right)^{1/q}
  \\
  \lesssim &
  \left(\sum_{k\in E}\int_{B(\xi_k,r)}a_k^q\lan x+\nabla\mu(\xi_k)/2\pi\ran^{-\scrL} d\xi\right)^{1/q}
  \\
  \lesssim &
  \left(\sum_{k\in E}a_k^q\lan x+\nabla\mu(\xi_k)/2\pi\ran^{-\scrL} \right)^{1/q}
  \\
  \lesssim &
  \sup_{k\in E}a_k\lan x+\nabla\mu(\xi_k)/2\pi\ran^{-\scrL}
  \left(\sum_{k\in E}\lan x+\nabla\mu(\xi_k)/2\pi\ran^{-\scrL} \right)^{1/q}.
\end{split}
\een
By Lemma \ref{lemma, scattered}, we know that
\be
|E_{0,x}|:=|\{k\in \zn: |x+\nabla\mu(\xi_k)/2\pi|\leq 1 \}|\lesssim 1
\ee
and
\be
|E_{j,x}|:=|\{k\in \zn: 2^{j-1}\leq |x+\nabla\mu(\xi_k)/2\pi|\leq 2^{j}\}|\lesssim 2^{jn},
\ee
uniformly for all $x\in \rn$ and $j\geq 1$.
From this, we have
\be
\begin{split}
  \sum_{k\in E}\lan x+\nabla\mu(\xi_k)/2\pi\ran^{-\scrL}
  \lesssim &
  \sum_{j=0}^{\fy}\sum_{k\in E_{j,x}}\lan x+\nabla\mu(\xi_k)/2\pi\ran^{-\scrL}
  \\
  \lesssim &
  \sum_{j=0}^{\fy}|E_{j,x}|2^{-jn\scrL}
  \lesssim
  \sum_{j=0}^{\fy}2^{jn}2^{-jn\scrL}\lesssim 1.
\end{split}
\ee
This and \eqref{pf, 3,10} imply that
\be
\begin{split}
  \left\|V_{\phi}G(x,\xi)\right\|_{L^q}
  \lesssim
  \sup_{k\in E}a_k\lan x+\nabla\mu(\xi_k)/2\pi\ran^{-\scrL}.
\end{split}
\ee
Then,
\be
\begin{split}
  \|G\|_{\wpq}
  = &
  \left\|\left\|V_{\phi}G(x,\xi)\right\|_{L^q_{\xi}}\right\|_{L^p_x}
  \\
  \lesssim &
  \left\|\sup_{k\in E}a_k\lan \cdot+\nabla\mu(\xi_k)/2\pi\ran^{-\scrL}\right\|_{L^p}
  \\
  \lesssim &
  \left\|\left(\sum_{k\in E}a_k^p\lan \cdot+\nabla\mu(\xi_k)/2\pi\ran^{-\scrL p}\right)^{1/p}\right\|_{L^p}
  \\
  = &
\left(\sum_{k\in E}a_k^p\left\|\lan \cdot+\nabla\mu(\xi_k)/2\pi\ran^{-\scrL}\right\|^p_{L^p_x}\right)^{1/p}
  \lesssim
  \left(\sum_{k\in E}a_k^p\right)^{1/p}.
\end{split}
\ee
Next, we consider the estimate of $\|F\|_{W^{p,q}_s}$.

\textbf{Lower estimate of $\|F\|_{\wpqs}$.}

By the definition of $h_k$ and $\phi$, we know that for $\xi\in B(\xi_l, r/4)$,
\be
\begin{split}
  \widehat{F}(\eta)\widehat{\phi}(\eta-\xi)
  = &
  \sum_{k\in E}a_k\widehat{h_k}(\eta)\widehat{\phi}(\eta-\xi)
  =
  a_l\widehat{h_l}(\eta).
\end{split}
\ee
Then for $\xi\in B(\xi_l, r/4)$,
\be
\begin{split}
  \left|V_{\phi}F(x,\xi)\right|
  = &
  \left|\int_{\rn}\widehat{F}(\eta)\widehat{\phi}(\eta-\xi)e^{2\pi i\eta\cdot x}d\eta\right|
  \\
  = &
  \left|a_l\int_{\rn}\widehat{h_l}(\eta)e^{2\pi i\eta\cdot x}d\eta\right|
  =a_l|h_l(x)|=a_l|h(x)|.
\end{split}
\ee
Form this we further deduce that
\be
\begin{split}
  \left\|V_{\phi}F(x,\xi)\right\|_{L^{q}_{\xi,s}}
  \geq &
  \left(\sum_{k\in E}\int_{B(\xi_k,r/4)}\left|V_{\phi}F(x,\xi)\right|^q
  \lan \xi\ran^{sq}d\xi\right)^{1/q}
  \\
  \sim &
  \left(\sum_{k\in E}a_k^q\lan \xi_k\ran^{sq}\right)^{1/q}|h(x)|
  \sim
  \left(\sum_{k\in E}a_k^q\lan k\ran^{sq/(\b-1)}\right)^{1/q}|h(x)|.
\end{split}
\ee
Then the desired estimate follows by
\be
\begin{split}
  \|F\|_{\wpqs}
  =
  \|\left\|V_{\phi}F(x,\xi)\right\|_{L^{q}_{\xi,s}}\|_{L^p_x}
  \gtrsim &
  \left(\sum_{k\in E}a_k^q\lan k\ran^{sq/(\b-1)}\right)^{1/q}\left\|h\right\|_{L^p}
  \sim \|\{a_k\}_{k\in E}\|_{l^q_{s/(\b-1)}}.
\end{split}
\ee

\textbf{Upper estimate of $\|F\|_{\wpqs}$.}
Next, we consider the estimate of $\|F\|_{W^{p,q}_s}$.
By the definition of $h_k$ and $\phi$, we know that for any $k\in A$,
\be
\{\xi\in \rn: \widehat{h_k}\widehat{\phi}(\eta-\xi)\neq 0\}\subset B(\xi_k, r),
\ee
where the family $\{B(\xi_k,r)\}_{k\in \zn}$ is pairwise disjoint.
For $\xi\in B(\xi_k,r)$,
\be
\begin{split}
  \left|V_{\phi}F(x,\xi)\right|
  = &
  \left|\int_{\rn}\widehat{F}(\eta)\widehat{\phi}(\eta-\xi)e^{2\pi i\eta\cdot x}d\eta\right|
  \\
  = &
  \left|a_k\int_{\rn}\widehat{h_k}(\eta)\widehat{\phi}(\eta-\xi)e^{2\pi i\eta\cdot x}d\eta\right|
  \\
  = &
  \left|a_k\int_{\rn}\widehat{h_k}(\eta)\widehat{M_{\xi}\phi}(\eta)e^{2\pi i\eta\cdot x}d\eta\right|
  \\
  = &
  a_k\left|h_k\ast M_{\xi}\phi(x)\right|.
\end{split}
\ee
For the last term, we further have
\be
\begin{split}
  a_k\left|h_k\ast M_{\xi}\phi(x)\right|
  \lesssim
  a_k|h_k|\ast |M_{\xi}\phi|(x)
  = a_k|h|\ast |\phi|(x)
  \lesssim
  a_k\langle x\rangle^{-\scrL},
\end{split}
\ee
where in the last inequality we use the fact that both $h$ and $\phi$ are Schwartz functions, $\scrL$ indicates a sufficiently large number.

Now, we have the following estimate
\be
\begin{split}
  \left\|V_{\phi}F(x,\xi)\right\|_{L^q_{s}}^q
  = &
  \sum_{k\in E}\int_{B(\xi_k,r)}\left|V_{\phi}F(x,\xi)\right|^q\lan \xi\ran^{sq}d\xi
  \\
  \lesssim &
  \sum_{k\in E}\int_{B(\xi_k,r)}a_k^q\langle x\rangle^{-q\scrL}\lan \xi\ran^{sq} d\xi
  \\
  \lesssim &
  \sum_{k\in E}a_k^q\lan \xi_k\ran^{sq} \langle x\rangle^{-q\scrL}
  \lesssim
  \sum_{k\in E}a_k^q\lan k\ran^{\frac{sq}{\b-1}}\langle x\rangle^{-q\scrL}.
\end{split}
\ee
Then,
\be
\begin{split}
  \|F\|_{\wpq}
  = &
  \left\|\left\|V_{\phi}F(x,\xi)\right\|_{L^{q}_{\xi,s}}\right\|_{L^p_x}
  \\
  \lesssim &
  \left\|\left(\sum_{k\in E}a_k^q\lan k\ran^{\frac{sq}{\b-1}}\langle \cdot\rangle^{-q\scrL}\right)^{1/q}\right\|_{L^p}
  \\
  = &
  \left\|\left(\sum_{k\in E}a_k^q\lan k\ran^{\frac{sq}{\b-1}}\right)^{1/q}\langle \cdot\rangle^{-\scrL}\right\|_{L^p}
  \sim \|\{a_k\}_{k\in E}\|_{l^q_{s/(\b-1)}}.
\end{split}
\ee
\end{proof}

\begin{lemma}[Rotation trick]\label{lemma, rotation}
  Let $\mu$ be a real-valued functions satisfying the assumptions of Theorem \ref{thm, wiener, sharp potential loss},
  For every nonnegative sequence $\{a_k\}_{k\in \zn}$, we have
  \be
  \|\{a_k\}_{k\in \zn}\|_{l^p}\lesssim \|\{a_k\}_{k\in \zn}\|_{l^q_{s/(\b-1)}},
  \ \
  \|\{a_k\}_{k\in \zn}\|_{l^q_{-s/(\b-1)}}\lesssim \|\{a_k\}_{k\in \zn}\|_{l^p}.
  \ee
\end{lemma}
\begin{proof}
  If $e^{i\mu(D)}: \wpqs\rightarrow \wpq$ is bounded, for any Schwartz function we have
  \be
  \|e^{i\mu(D)}f\|_{\wpq}\lesssim \|f\|_{\wpqs},\ \ \
  \|f\|_{W^{p,q}_{-s}}\lesssim \|e^{i\mu(D)}f\|_{\wpq}.
  \ee
  This and Lemma \ref{lemma, estimates of special} yield that
  \be
  \|\{a_k\}_{k\in E}\|_{l^p}\lesssim \|\{a_k\}_{k\in E}\|_{l^q_{s/(\b-1)}},
  \ \
  \|\{a_k\}_{k\in E}\|_{l^q_{-s/(\b-1)}}\lesssim \|\{a_k\}_{k\in E}\|_{l^p}.
  \ee
In the above two inequalities, the set $E$ will be replaced by $\zn$, by using a rotation trick as follows.

  Denote by $\mu_P(\xi):=\mu(P^{-1}\xi)$,
  where $P$ is a orthogonal matrix.
  By a direct calculation we get
  \be
  \begin{split}
    e^{i\mu_P(D)}f=(e^{i\mu(D)}f_{P^{-1}})_P
  \end{split}
  \ee
  where $f_{P^{-1}}(x):=f(Px)$.
  Using Lemma \ref{lemma, rotation free}, we get
  \be
  \begin{split}
    \|e^{i\mu_P(D)}f\|_{\wpq}
    = &
    \|(e^{i\mu(D)}f_{P^{-1}})_P\|_{\wpq}
    \\
    \sim &
    \|e^{i\mu(D)}f_{P^{-1}}\|_{\wpq}
    \lesssim
    \|f_{P^{-1}}\|_{\wpqs}\sim \|f\|_{\wpqs}.
  \end{split}
  \ee
  Now, we have proved that the operator $e^{i\mu_P(D)}: \wpqs\rightarrow \wpq$ is bounded uniformly for all orthogonal matrix $P$.
  A direct calculation yields that
  \be
  \text{Hess}\mu_P(\xi)=P\text{Hess}\mu(P^{-1}\xi)P^{-1}.
  \ee
  Then $\text{Hess}\mu_P(\xi)$ and $\text{Hess}\mu(P^{-1}\xi)$ have the same eigenvalues.
  From the above arguments we claim that
  $\mu_P$ satisfies all the assumption of Theorem \ref{thm, wiener, sharp potential loss}
  when $\k_0$ is replaced by $P\k_0$.

Moreover, we can apply the same argument of Lemma \ref{lemma, scattered} and \ref{lemma, estimates of special} to the new operator
  $e^{i\mu_P}$ and get
  \ben\label{pf, 3,9}
    \|\{a_k\}_{k\in E_P}\|_{l^p}\lesssim \|\{a_k\}_{k\in E_P}\|_{l^q_{s/(\b-1)}},
  \
  \|\{a_k\}_{k\in E_P}\|_{l^q_{-s/(\b-1)}}\lesssim \|\{a_k\}_{k\in E_P}\|_{l^p}.
  \een
  where
\be
E_P:= \{l\in \bbZ^n\bs\{0\}: \langle l\rangle^{\frac{2-\b}{\b-1}}l\in P\G\},\ \xi_l:=\langle l\rangle^{\frac{2-\b}{\b-1}}l,
\ee
$\G$ is the cone chosen in the proof of Lemma \ref{lemma, scattered}.
Note that there exist finite orthogonal matrix, denoted by $P_i$ such that $\bigcup E_{P_i}=\zn\bs\{0\}$.
From this and \eqref{pf, 3,9}, we get
\be
 \begin{split}
   \|\{a_k\}_{k\in \zn}\|_{l^p}
   \lesssim &
   \sum_i\|\{a_k\}_{k\in E_{p_i}}\|_{l^p}+|a_0|
   \\
   \lesssim &
   \sum_i\|\{a_k\}_{k\in E_{P_i}}\|_{l^q_{s/(\b-1)}}+|a_0|
 \lesssim \|\{a_k\}_{k\in \zn}\|_{l^q_{s/(\b-1)}}.
 \end{split}
\ee
A similar argument yields another desired conclusion:
\be
\|\{a_k\}_{k\in \zn}\|_{l^q_{-s/(\b-1)}}\lesssim \|\{a_k\}_{k\in \zn}\|_{l^p}.
\ee
\end{proof}

\begin{proof}[Proof of Theorem \ref{thm, wiener, sharp potential loss}]
  If $p\leq q$,
  By Lemma \ref{lemma, rotation}, we have
  \be
  l^q_{s/(\b-1)}\subset l^p.
  \ee
  From this and Lemma \ref{lemma, discrete embedding}, we further obtain
  \be
  1/p\leq 1/q+s/n(\b-1)\Longleftrightarrow s\geq n(\b-1)(1/p-1/q)=n(\b-1)|1/p-1/q|
  \ee
  with strict inequality when $p<q$.

  If $p>q$, we use Lemma \ref{lemma, rotation} to get
  \be
  l^p \subset l^q_{-s/(\b-1)}.
  \ee
  Then Lemma \ref{lemma, discrete embedding} further imply that
  \be
   1/q-s/n(\b-1)<1/p\Longleftrightarrow s> n(\b-1)(1/q-1/p)=n(\b-1)|1/p-1/q|.
  \ee
\end{proof}

\begin{proof}[Proof of Corollary \ref{coy, wiener, sharp potential loss}]
  First, the sufficiency follows by Corollary \ref{coy, bd, exact conditions}.

  Note that $e^{i|D|^{\b}}$ and $e^{-i|D|^{\b}}$ are bounded on $\wpq$ when $\b\in (0,1]$.
  If $e^{i|D|^{\b}}: W_{p,q}^{\d}\rightarrow \wpq$ is bounded, we have

  \be
  \|f\|_{\wpq}=\|e^{-i|D|^{\b}}(e^{i|D|^{\b}}f)\|_{\wpq}
  \lesssim
  \|e^{i|D|^{\b}}\|_{\wpq} \lesssim \|f\|_{W_{p,q}^{\d}},
  \ee
  which implies that $\d\geq 0$.
  Finally, when $\b\in (1,2]$, by Lemma \ref{lemma, derivative to wiener} we have
  $(1-\r_0)\partial^{\g}\mu \in \calC^{[n/(\dot{p}\wedge \dot{q})]+1}\subset W^{\fy,\fy}_{[n/(\dot{p}\wedge \dot{q})]+1}
  \subset  W^{\infty,\fy}_{n/(\dot{p}\wedge \dot{q})+\ep}\ (|\g|=2)$ for some $\ep>0$, where $\r_0$ is the function
  mentioned in Theorem \ref{thm, wiener, sharp potential loss}.
  Then the necessity follows by Theorem \ref{thm, wiener, sharp potential loss}.
\end{proof}

Finally, we give the proof for the claim in Remark \ref{rk, plenty funtions}.
By the smoothness and homogeneous property of $\mu$, we have
\be
(1-\r_0)\partial^{\g}\mu \in \calC^{[n/(\dot{p}\wedge \dot{q})]+1}\subset W^{\fy,\fy}_{[n/(\dot{p}\wedge \dot{q})]+1}
  \subset  W^{\infty,\fy}_{n/(\dot{p}\wedge \dot{q})+\ep}\ (|\g|=2)
\ee
for some $\ep>0$. Next, we show there exists a point on $\mathbb{S}^{n-1}$ such that the Hessian matrix of $\mu$ is no-degenerate.
Here, we follows the argument in \cite[Appendix A]{Miyachi2009PAMS} with slight modification.
Without loss of generality we assume that $\mu(\xi)>0$ for all $\xi\neq 0$.
Write
\be
\mu(\xi)=|\xi|^{\b}\th(\xi),
\ee
where $\th$ is a homogeneous function of degree $0$ with $\th(\xi)>0 (\xi\neq 0)$.
A direct calculation yields that
\be
\begin{split}
  \frac{\partial^2\mu(\xi)}{\partial \xi_i\partial \xi_j}
  = &
  \b|\xi|^{\b-2}\d_{ij}\th(\xi)+\b(\b-2)|\xi|^{\b-4}\xi_i\xi_j\th(\xi)
  \\
  & +\frac{\partial|\xi|^{\b}}{\partial \xi_i}\frac{\partial \th(\xi)}{\partial\xi_j}
  +\frac{\partial|\xi|^{\b}}{\partial \xi_j}\frac{\partial \th(\xi)}{\partial \xi_i}
  +|\xi|^{\b}\frac{\partial^2\th(\xi)}{\partial \xi_i\partial \xi_j}.
\end{split}
\ee
There exists a point $\xi_0\in \mathbb{S}^{n-1}$ such that $\th(\xi_0)$ takes its minimum in $\rn\bs\{0\}$.
Thus, $\nabla\th(\xi_0)=0$,
and $\text{Hess}\th(\xi_0)$ is nonnegative definite. From this and the above equality we have
\be
\begin{split}
  \frac{\partial^2\mu(\xi_0)}{\partial \xi_i\partial \xi_j}
  = &
  \b\d_{ij}\th(\xi_0)+\b(\b-2)\xi_{0,i}\xi_{0,j}\th(\xi_0)
  +\frac{\partial^2\th(\xi_0)}{\partial x_i\partial x_j}.
\end{split}
\ee
Hence,
\be
\th(\xi_0)^{-1}\text{Hess}\mu(\xi_0)=\b E+\b(\b-2)(\xi_{0,i}\xi_{0,j})_{i,j}+\th(\xi_0)^{-1}\text{Hess}\th(\xi_0)= : A+\th(\xi_0)^{-1}\text{Hess}\th(\xi_0)
\ee
For a vector $x\in \rn$ we get
\be
\begin{split}
  x^TAx
  = &\b|x|^2+\b(\b-2)|\xi_0\cdot x|^2
  \\
  \geq &
  \b|x|^2+\b(\b-2)|\xi_0|^2\cdot |x|^2
  \\
  = &
  \b|x|^2+\b(\b-2)|x|^2=\b(\b-1)|x|^2.
\end{split}
\ee
Recall $\b\in (1,2]$, then the matrix $A$ is positive definite.
This implies that  $\text{Hess}\mu=A+\th(\xi_0)^{-1}\text{Hess}\th(\xi_0)$ is positive definite, which yields the non-degenerate of $\text{Hess}\mu(\xi_0)$.

\section{Complements: high growth of $\mu$}
Keep the prototype $\mu(\xi)=|\xi|^{\b}$ under consideration,
in order to get the desired boundedness result for the high growth case $\b>2$,
we find that the previous working space $W^{\fy,\fy}_{n/(\dot{p}\wedge \dot{q})+\ep}$ should be replaced
by a more reasonable one.
In fact, in the high growth case, in order to use the information of the second-order derivatives of $\mu$,
the working space is expected to be a function space in which the
functions can not only be localized in time, but also be invariant under the modulation operator.
By this observation, the Wiener amalgam space without potential, such like $\wpq$, may be a good choice.

On the other hand,
to establish the boundedness result on Wiener amalgam spaces,
another approach as used in \cite{Cunanan2014JMAA}, is to apply the boundedness result on modulation spaces.
We also note that the the natural working space for modulation case is just a Wiener amalgam space without potential,
for instance one can see the natural working space $W^{\fy,1}$ used in \cite{NicolaTabacoo2018JPDOA}.
Hence, in the high growth case,
it is a reasonable choice to verify the boundness result on Wiener amalgam spaces with the aid of
corresponding boundedness result on modulation spaces.
Here, we first give a generalization
of Theorem 1.2 in \cite{NicolaTabacoo2018JPDOA},
then verify the boundedness results on Wiener amalgam spaces
by an embedding relations between modulation and Wiener amalgam spaces.

\begin{lemma}\label{thm, wiener, high growth}\label{lemma, modulation case, high growth}
  Suppose $0<p,q\leq \infty$.
  Let $\mu$ be a real-valued $C^2(\bbR^n)$ function satisfying
  \be
  \begin{cases}
    \lan\xi\ran^{-s}\partial^{\g}\mu \in W^{\infty,1}\ (|\g|=2),\ \  \text{if}\  p\geq 1;
    \\
    \lan\xi\ran^{-s}\partial^{\g}\mu \in W^{\infty,1}_{n(1/\dot{p}-1)+\ep}\ (|\g|=2),\ \  \text{if}\  p<1,
  \end{cases}
  \ee
  for some $s,\ep>0$.
  Then $e^{i\mu(D)}: M^{p,q}_{\d} \rightarrow \mpq$ is bounded for $\d\geq sn|1/p-1/2|$.
\end{lemma}
\begin{proof}
    We only give the sketch of this proof, since it is similar as the proof of Theorem \ref{thm, wiener, potential loss}
    and Lemma \ref{lemma, modulation case}.
    By the convolution relation (see \cite[Corollary 4.2]{GuoChenFanZhao2018MJM})
  \be
  M^{p,q}_{\d} \ast M^{\dot{p},\fy}_{-\d}\subset \mpq,
  \ee
  we only need to verify that $\scrF^{-1}e^{i\mu}\in M^{\dot{p},\fy}_{-\d}$,
  or equivalently, $\lan \xi\ran^{-\d}e^{i\mu}\in W^{\fy,\dot{p}}$.
  Write
\be
\|P_{-\d}e^{i\mu}\|_{W^{\fy,\dot{p}}}\sim \sup_{k\in \bbZ^n}\|\s_k P_{-\d}e^{i\mu}\|_{W^{\fy,\dot{p}}}
=\sup_{k\in \bbZ^n}\lan k\ran^{-\d}\|\s_k e^{i\mu}\|_{W^{\fy,\dot{p}}}.
\ee
Set
  \be
  B_k:=\{l\in \bbZ^n:\   \s_l\cdot \s_k(\frac{\cdot}{\langle k\rangle^{s/2}})\neq 0\}.
  \ee
  Observe $|B_k|\sim \lan k\ran^{sn/2}$,
  and recall that $W^{\fy,\dot{p}}$ is a Banach algebra (see \cite[Corollary 4.2]{GuoChenFanZhao2018MJM}).
  We have
  \be
  \begin{split}
   \|\s_ke^{i\mu}\|_{W^{\fy,\dot{p}}}
    = &
    \|\s_k\sum_{l\in B_k}\s_l(\langle k\rangle^{s/2}\cdot)e^{i\mu}\|_{W^{\fy,\dot{p}}}
    \\
    \lesssim &
    \|\s_k\|_{W^{\fy,\dot{p}}}
    \|\sum_{l\in B_k}\s_l(\langle k\rangle^{s/2}\cdot)e^{i\mu}\|_{W^{\fy,\dot{p}}}
    \\
    \lesssim &
    \|\sum_{l\in B_k}\s_l(\langle k\rangle^{s/2}\cdot)e^{i\mu}\|_{W^{\fy,\dot{p}}}
    \\
    \lesssim &
    \left(\sum_{l\in B_k}\|\s_l(\langle k\rangle^{s/2}\cdot)e^{i\mu}\|^{\dot{p}}_{W^{\fy,\dot{p}}}\right)^{1/\dot{p}}
    \lesssim \lan k\ran^{\frac{sn}{2\dot{p}}}\sup_{l\in B_k}\|\s_l(\langle k\rangle^{s/2}\cdot)e^{i\mu}\|_{W^{\fy,\dot{p}}}
  \end{split}
  \ee
  Denote by $\mu_k(x):=\mu(\frac{x}{\langle k\rangle^{s/2}})$, and
  \be
  R_l^k(\xi):=\mu_k(\xi+l)-\mu_k(l)-\nabla \mu(l)\xi
=\sum_{|\g|=2}\frac{2\xi^{\g}}{\g!}\int_0^1(1-t)\partial^{\g}\mu(l+t\xi)dt.
\ee
  For $l\in B_k$, we further have
  \be
  \begin{split}
    \|\s_l(\langle k\rangle^{s/2}\cdot)e^{i\mu}\|_{W^{\fy,\dot{p}}}
    = &
    \|\s_l(\langle k\rangle^{s/2}\cdot)e^{i\mu}\|_{\scrF L^{\dot{p}}}
    \\
    = &
    \lan k\ran^{\frac{sn}{2}(1/\dot{p}-1)}
    \|\s_le^{i\mu_k}\|_{\scrF L^{\dot{p}}}
    \\
    = &
    \lan k\ran^{\frac{sn}{2}(1/\dot{p}-1)}
    \|\s_0e^{i\mu_k(\cdot+l)}\|_{\scrF L^{\dot{p}}}
    \\
    = &
    \lan k\ran^{\frac{sn}{2}(1/\dot{p}-1)}
    \|\s_0e^{iR_l^k}\|_{\scrF L^{\dot{p}}}
    \\
    \sim &
    \lan k\ran^{\frac{sn}{2}(1/\dot{p}-1)}
    \|\s_0e^{iR_l^k}\|_{W^{\fy,\dot{p}}}
    \\
    \lesssim &
    \lan k\ran^{\frac{sn}{2}(1/\dot{p}-1)}\exp(C\|\s_0^*R_l^k\|_{W^{\fy,\dot{p}}})
    \lesssim
    \lan k\ran^{\frac{sn}{2}(1/\dot{p}-1)},
  \end{split}
  \ee
where in the last inequality
we use $\|\s_0^*R_l^k\|_{W^{\fy,\dot{p}}}=\|\s_0^*R_l^k\|_{W^{\fy,1}}\lesssim 1$ for $p\geq 1$, and use
$\|\s_0^*R_l^k\|_{W^{\fy,\dot{p}}}\lesssim \|\s_0^*R_l^k\|_{W^{\infty,1}_{n(1/\dot{p}-1)+\ep}}\lesssim 1$ for $p<1$,
which can be derived by a similar argument as in the proof of Theorem \ref{thm, wiener, potential loss}.

Combining the above estimates yields that for $p=\fy$ or $p\leq 1$,
\be
\begin{split}
\|P_{-\d}e^{i\mu}\|_{W^{\fy,\dot{p}}}
\sim &\sup_{k\in \bbZ^n}\lan k\ran^{-\d}\|\s_ke^{i\mu}\|_{W^{\fy,\dot{p}}}
\\
\lesssim &
\sup_{k\in \bbZ^n}\lan k\ran^{-\d}\lan k\ran^{\frac{sn}{2\dot{p}}}\sup_{l\in B_k}\|\s_l(\langle k\rangle^{s/2}\cdot)e^{i\mu}\|_{W^{\fy,\dot{p}}}
\\
\lesssim &
\sup_{k\in \bbZ^n}\lan k\ran^{-\d}\lan k\ran^{\frac{sn}{2\dot{p}}}\lan k\ran^{\frac{sn}{2}(1/\dot{p}-1)}
\sim
\sup_{k\in \bbZ^n}\lan k\ran^{-\d+sn(1/\dot{p}-1/2)}\lesssim 1,
\end{split}
\ee
where we use the assumption $\d\geq sn|1/p-1/2|=sn(1/\dot{p}-1/2)$ as $p=\fy$ or $p\leq 1$.

Note that when $p=2$, $e^{i\mu(D)}$ is bounded on $\mpq$. The final conclusion then follows by an interpolation
among the cases of $p=2$, $p=\fy$ and $p\leq 1$.
\end{proof}

Then, we recall an embedding relations between modulation and Wiener amalgam spaces.
The proof is easy, so we omit here.
\begin{lemma}\label{lemma, emb between modulation and Wiener}
  Let $p,q\in (0,\fy]$, then
  \be
  \begin{split}
    M^{p,q}\subset W^{p,q},\ \text{if}&\ p\geq q,
    \\
    M^{p,q}_{\d}\subset W^{p,q},\ \text{if}&\ p< q, \d>n|1/p-1/q|,
  \end{split}
  \ee
  and
  \be
    \begin{split}
    W^{p,q}\subset M^{p,q},\ \text{if}&\ p\leq q,
    \\
    W^{p,q}_{\d}\subset M^{p,q},\ \text{if}&\ p> q, \d>n|1/p-1/q|.
  \end{split}
  \ee
\end{lemma}

Now, we are in a position to give our desired conclusion.
\begin{theorem}[high growth of $\mu$]\label{thm, wiener, potential loss, high growth}
  Suppose $0<p,q\leq \infty$.
  Let $\mu$ be a real-valued $C^2(\bbR^n)$ function satisfying
  \be
  \begin{cases}
    \lan\xi\ran^{-s}\partial^{\g}\mu \in W^{\infty,1},\ \  \text{if}\  p\geq 1;
    \\
    \lan\xi\ran^{-s}\partial^{\g}\mu \in W^{\infty,1}_{n(1/\dot{p}-1)+\ep},\ \  \text{if}\  p<1,
  \end{cases}
  \ee
  for some $s,\ep>0$ and all $|\g|=2$.
  Then $e^{i\mu(D)}: W^{p,q}_{\d} \rightarrow \wpq$ is bounded for $\d\geq sn|1/p-1/2|+n|1/p-1/q|$ with
  strict inequality when $p\neq q$.
\end{theorem}
\begin{proof}
  The case $p=q$ follows by Lemma \ref{lemma, modulation case, high growth} and the fact $W^{p,p}=M^{p,p}$.

  For $p>q$, we use Lemma \ref{lemma, modulation case, high growth} and \ref{lemma, emb between modulation and Wiener} to deduce
  \be
  \|e^{i\mu(D)}f\|_{\wpq}
  \lesssim \|e^{i\mu(D)}f\|_{\mpq}
  \lesssim \|f\|_{M^{p,q}_{sn|1/p-1/2|}}
  \lesssim \|f\|_{W^{p,q}_{\d}}.
  \ee

  For $p<q$, we use Lemma \ref{lemma, modulation case, high growth} and \ref{lemma, emb between modulation and Wiener} to deduce
  \be
  \|e^{i\mu(D)}f\|_{\wpq}
  \lesssim \|e^{i\mu(D)}f\|_{M^{p,q}_{\d-sn|1/p-1/2|}}
  \lesssim \|f\|_{M^{p,q}_{\d}}
  \lesssim \|f\|_{W^{p,q}_{\d}}.
  \ee
\end{proof}

As in Corollary \ref{coy, bd, exact conditions}, we can also establish the conclusions
fitting more detailed derivative conditions of $\mu$. Obviously, following corollary is an improvement of \cite[Theorem 1.1]{Cunanan2014JMAA}
\begin{corollary}\label{coy, bd, exact conditions, high growth}
  Suppose $0<p,q\leq \infty$.
  Let $\ep>n(1/\dot{p}-1)$, $\b>2$.
  Let $\mu$ be a real-valued function of class $C^{[n/(\dot{p}\wedge \dot{q})]+3}$ on $\mathbb{R}^n \backslash \{0\}$ which satisfies
\begin{equation}
|\partial^{\gamma}\mu(\xi)|\leq C_{\gamma}|\xi|^{\epsilon-|\gamma|}, \hspace{5mm} 0<|\xi|\leq 1,~|\gamma|\leq[n/(1/\dot{p}-1/2)]+1,
\end{equation}
and
\begin{equation}
  |\partial^{\gamma}\mu(\xi)|\leq C_{\gamma}|\xi|^{\beta-2}, \hspace{5mm} |\xi|>1,~2\leq|\gamma|\leq[n(1/\dot{p}-1/2)]+3
\end{equation}
Then $e^{i\mu(D)}: W^{p,q}_{\d} \rightarrow \wpq$ is bounded for $\d\geq n(\b-2)|1/p-1/2|+n|1/p-1/q|$ with
  strict inequality when $p\neq q$.
\end{corollary}
\begin{proof}[Proof of Corollary \ref{coy, bd, exact conditions, high growth}]
  Let $\r_0$ be a smooth function supported on B(0,1) and
  satisfies $\r_0(\xi)=1$ on $B(0,1/2)$. Denote by
  \be
  \mu_1:= \r_0\mu,\ \ \ \mu_2:= (1-\r_0)\mu.
  \ee
  The boundedness of $e^{i\mu_1(D)}$ follows by the same argument as in the proof of Corollary \ref{coy, bd, exact conditions}.
  Now we turn to the estimate of $e^{i\mu_2(D)}$.

  Denote $s:=\b-2$.
  We claim that
  \be
  \begin{cases}
    \lan\xi\ran^{-s}\partial^{\g}\mu \in W^{\infty,1},\ \  \text{if}\  p\geq 1;
    \\
    \lan\xi\ran^{-s}\partial^{\g}\mu \in W^{\infty,1}_{n(1/\dot{p}-1)+\ep},\ \  \text{if}\  p<1,
  \end{cases}
  \ee
  for some $\ep>0$ and all $|\g|=2$.
  Then the final conclusion follows by this claim and Theorem \ref{thm, wiener, potential loss, high growth}.

  For any fixed $\g_0$ with $|\g_0|=2$, we denote $g_{\g_0}:=\lan\xi\ran^{-s}\partial^{\g_0}\mu_2$.
  It follows by the assumption that
  \begin{equation}
  |\partial^{\g}g_{\g_0}(\xi)|\lesssim 1, \hspace{5mm} |\xi|>1,~|\gamma|\leq [n(1/\dot{p}-1/2)]+1.
\end{equation}
  Thus, for $p\geq 1$, we use the H\"{o}lder inequality to deduce that
  \be
  \begin{split}
    \|g_{\g_0}\|_{W^{\fy,1}}= &\sup_{k\in \zn}\|\s_kg_{\g_0}\|_{\scrF L^1}
    \\
    = &
    \sup_{k\in \zn}\|\scrF^{-1}(\s_kg_{\g_0})\|_{L^1}
    \\
    \lesssim &
    \sup_{k\in \zn}\|\lan x\ran^{[n/2]+1}\scrF^{-1}(\s_kg_{\g_0})(x)\|_{L^2}
    \\
    \sim &
    \sup_{k\in \zn}\|\s_kg_{\g_0}\|_{H^{[n/2]+1}}\sim \sup_{k\in \zn}\sum_{|\g|\leq [n/2]+1}\|\partial^{\g}(\s_kg_{\g_0})\|_{L^2}
    \lesssim 1.
  \end{split}
  \ee
  For $p< 1$, choose $\ep$ such that $n(1/\dot{p}-1/2)+\ep<[n(1/\dot{p}-1/2)]+1$, then
    \be
  \begin{split}
    &\|g_{\g_0}\|_{W^{\infty,1}_{n(1/\dot{p}-1)+\ep}}
    = \sup_{k\in \zn}\|\s_kg_{\g_0}\|_{\scrF L^{1}_{n(1/\dot{p}-1)+\ep}}
    \\
    = &
    \sup_{k\in \zn}\|\lan x\ran^{n(1/\dot{p}-1)+\ep}\scrF^{-1}(\s_kg_{\g_0})(x)\|_{L^{1}}
    \\
    = &
    \sup_{k\in \zn}\|\lan x\ran^{(n(1/\dot{p}-1/2)+\ep)-([n(1/\dot{p}-1/2)]+1)-n/2}
    \lan x\ran^{[n(1/\dot{p}-1/2)]+1}\scrF^{-1}(\s_kg_{\g_0})(x)\|_{L^{1}}
    \\
    \lesssim &
    \sup_{k\in \zn}\|\lan x\ran^{[n(1/\dot{p}-1/2)]+1}\scrF^{-1}(\s_kg_{\g_0})(x)\|_{L^2}
    \\
    \sim &
    \sup_{k\in \zn}\|\s_kg_{\g_0}\|_{H^{[n(1/\dot{p}-1/2)]+1}}\sim \sup_{k\in \zn}\sum_{|\g|\leq [n(1/\dot{p}-1/2)]+1}\|\partial^{\g}(\s_kg_{\g_0})\|_{L^2}
    \lesssim 1.
  \end{split}
  \ee
  We have now verified the claim and completed this proof.
\end{proof}

Applying Corollary \ref{coy, bd, exact conditions} and \ref{coy, bd, exact conditions, high growth} to the prototype $\mu(\xi)=|\xi|^{\b}$, we deduce the following conclusion.
\begin{corollary}
  Let $0<p,q\leq \fy$, $\b>n(1/\dot{p}-1)$.
  We have $e^{i|D|^{\b}}: W_{p,q}^{\d}\rightarrow \wpq$ is bounded if
  \be
  \d\geq n|1/p-1/q|\min\{\max\{\b-1,0\},1\}+n|1/p-1/2|\max\{\b-2,0\}
  \ee
  with strict inequality when $\b>1, p\neq q$.
\end{corollary}

\begin{remark}
  As one can see, in the high growth case $\b>2$, the potential loss comes from two aspects.
  The first one $n|1/p-1/q|$ can be viewed as the result of the scattered property of $\nabla\mu$,
  and the second one $n|1/p-1/2|(\b-2)$, which is vanish in the low and mild growth cases $0<\b\leq 2$, comes from the second-order derivative of $\mu$ as in
  $e^{i\mu(D)}: M^{p,q}_{\d}\rightarrow\mpq$ boundedness case.
  Thus, in the high growth case,
  the more complex composition of potential loss may lead to greater difficulty of determining the sharp loss of potential.
  For this direction, a partial result can be found in \cite{Cunanan2014JMAA}.
  \end{remark}
  \noindent\textbf{Theorem D (\cite[Theorem 1.2]{Cunanan2014JMAA}).}\  Suppose $1\leq p,q\leq \fy$.
  Let $\b\geq 2$ and $\mu$ be a real-valued $C^{\fy}(\rn\bs\{0\})$ function which is homogeneous of order $\b$.
  Suppose that there exists a point $\xi_0\neq 0$ at which the Hessian determinant of $\mu$ is not zero.
  Let $\max\{1/q,1/2\}\leq 1/p$ or $\min\{1/q,1/2\}\geq 1/p$, $s\in \mathbb{R}$.
  Then the boundedness of $e^{i\mu(D)}: W^{p,q}_{s} \rightarrow \wpq$ implies
  \be
  s\geq n|1/p-1/q|+n|1/p-1/2|(\b-2).
  \ee
 \\~

  Note that when $\b=2$, by Theorem \ref{thm, wiener, sharp potential loss}, the conclusion in Theorem D can be improved to
  \be
  s> n|1/p-1/q|+n|1/p-1/2|(\b-2)
  \ee
  when $p\neq q$ in the full range $0<p,q\leq \fy$. However, in the high growth case $\b>2$,
  it still leaves an open problem whether $\d=n|1/p-1/q|+n|1/p-1/2|(\b-2)$ is the optimal potential loss index in the
  boundedness result $e^{i\mu(D)}: W^{p,q}_{\d} \rightarrow \wpq$ when $p\neq q$.

\appendix
\section{}
In order to prove Lemma \ref{lemma, conv of Wfy}, we first recall the sharp version of Young's inequality of discrete form.
\begin{lemma}[Lemma 4.5 in \cite{GuoChenFanZhao2018MJM}]\label{lemma, sharpness of the Young's inequality, discrete form}
Suppose $0<q,q_1,q_2 \leq \infty$. Set $S:=\{j\in \mathbb{Z}: \ q_j\geq 1, 1\leq j\leq 2\}$.
Then
\be
l_{q_1}\ast l_{q_2}\subset l_{q}
\ee
holds if and only if
\be
\begin{cases}
(|S|-1)+1/q\leq 1/q_1+1/q_2,
\\
1/q\leq 1/q_1,\ 1/q\leq 1/q_2.
\end{cases}
\ee
\end{lemma}
By this lemma, we further have following useful inequality.
\begin{lemma}\label{lemma, conv under 1}
  Let $0<p,q\leq \fy$. We have
  \be
  l^{\frac{q}{\dot{p}}}\ast l^{\frac{\dot{p}\wedge \dot{q}}{\dot{p}}}\subset l^{\frac{q}{\dot{p}}}
  \ee
\end{lemma}
\begin{proof}
  Denote by $r_1=r:=\frac{q}{\dot{p}}$, $r_2:=\frac{\dot{p}\wedge \dot{q}}{\dot{p}}$. We have
  \ben\label{pf, A1}
  1/r\leq 1/r_1,\  1/r\leq 1/r_2.
  \een
  We divide this proof into two cases.

  Case 1: $r< 1$ or $r_2< 1$. The desired conclusion follows by Lemma \ref{lemma, sharpness of the Young's inequality, discrete form} and \eqref{pf, A1}.

  Case 2: $r, r_2\geq 1$. We only need to check
  \be
  1+1/r\leq 1/r_1+1/r_2,
  \ee
  which is equivalent to
  \be
  1\leq 1/r_2\Longleftrightarrow \frac{\dot{p}\wedge \dot{q}}{\dot{p}}\leq 1.
  \ee

\end{proof}

Then, we give the following product relation on modulation space.
\begin{lemma}\label{lemma, product on modulation}
  Let $0<p,q\leq \fy$, we have
  \be
  \mpq\cdot M^{\fy,\dot{p}\wedge \dot{q}} \subset \mpq.
  \ee
\end{lemma}
\begin{proof}
  Using the almost
orthogonality of the frequency projections $\sigma _{k}$, for all $k\in \mathbb{Z}^{n}$ we have
\begin{equation*}
\Box_k(fg)=\sum_{i,j\in \mathbb{Z}^n}\Box_k(\Box_if\cdot \Box_jg)
=\sum_{|l|\leq c(n)}\sum_{i+j=k+l}\Box_k(\Box_if\cdot \Box_jg),
\end{equation*}
where \ $c(n)$ \ is a constant depending only on \ $n.$
By the fact that $\Box_k$ is an $L^p$ multiplier, we have
\be
\begin{split}
\|\Box_k(fg)\|_{L^p}
= & \|\sum_{|l|\leq c(n)}\sum_{i+j=k+l}\Box_k(\Box_if\cdot \Box_jg)\|_{L^p}
\\
\lesssim &
\left(\sum_{|l|\leq c(n)}\sum_{i+j=k+l}\|\Box_k(\Box_if\cdot \Box_jg)\|^{\dot{p}}_{L^p}\right)^{1/\dot{p}}
\\
\lesssim &
\left(\sum_{|l|\leq c(n)}\sum_{i+j=k+l}\|\Box_if\cdot \Box_jg\|^{\dot{p}}_{L^p}\right)^{1/\dot{p}}
\\
\leq &
\left(\sum_{|l|\leq c(n)}\sum_{i+j=k+l}\|\Box_if\|^{\dot{p}}_{L^{p}}\|\Box_jg\|^{\dot{p}}_{L^{\fy}}\right)^{1/\dot{p}}.
\end{split}
\ee
By the definition of modulation space, we further have
\ben\label{pf, A2}
\begin{split}
\|fg\|^{\dot{p}}_{\mpq}
= &
\big\|\{\|\Box_k(fg)\|_{L^p}\}\big\|^{\dot{p}}_{l^q}
\\
\lesssim &
\big\|\{\|\Box_if\|^{\dot{p}}_{L^{p_1}}\}\ast \{\|\Box_jg\|^{\dot{p}}_{L^{p_2}}\}\big\|_{l^{q/\dot{p}}}
\\
\lesssim &
\big\|\{\|\Box_if\|_{L^{p_1}}\}\big\|^{\dot{p}}_{l^{q}}
\big\|\{\|\Box_jg\|_{L^{p_2}}\}\big\|_{l^{\dot{p}\wedge \dot{q}}}^{\dot{p}},
\end{split}
\een
where in the last inequality we use the convolution relation
$
l^{\frac{q}{\dot{p}}}\ast l^{\frac{\dot{p}\wedge \dot{q}}{\dot{p}}}\subset l^{\frac{q}{\dot{p}}}.
$
in Lemma \ref{lemma, conv under 1}.
The final conclusion follows by \eqref{pf, A2}.
\end{proof}

Now, we are in a position to give the proof of Lemma \ref{lemma, conv of Wfy}.
\begin{proof}[Proof of Lemma \ref{lemma, conv of Wfy}]
The second conclusion $W^{\fy,\fy}_{s}\cdot W^{\fy,\fy}_{s}\subset W^{\fy,\fy}_{s}$
can be found in \cite[Corollary 4.2]{GuoChenFanZhao2018MJM}.

Following, we focus on the proof of $W^{p,q}_{\d}\ast W^{\dot{p}\wedge \dot{q},\fy}_{-\d}\subset W^{p,q}$.
For $f\in W^{p,q}_{\d}$ and $g\in W^{\dot{p}\wedge \dot{q},\fy}_{-\d}$,  denote by $F=\lan D\ran^{\d}f$
and $G=\lan D\ran^{-\d}g$. Then
\be
\|f\|_{W^{p,q}_{\d}}\sim \|F\|_{\wpq},\ \ \
\|g\|_{W^{\dot{p}\wedge \dot{q},\fy}_{-\d}}\sim \|G\|_{W^{\dot{p}\wedge \dot{q},\fy}},
\ \ \ f\ast g=F\ast G.
\ee
From this, we only need to verify the following equivalent relation
\be
W^{p,q}\ast W^{\dot{p}\wedge \dot{q},\fy}\subset W^{p,q},
\ee
which is equivalent to the following product inequality on modulation space
\be
M^{q,p}\cdot M^{\fy, \dot{p}\wedge \dot{q}}\subset M^{q,p}.
\ee
This is just the conclusion in Lemma \ref{lemma, product on modulation}.
\end{proof}

\subsection*{Acknowledgements}
Guo's work was partially supported by the National Natural Science Foundation of China (Nos. 11701112, 11671414),
and
Zhao's work was partially supported by the National Natural Science Foundation of China (Nos. 11601456, 11771388).

\end{document}